\documentclass[12pt]{article}
\usepackage{amsmath}
\usepackage[usenames]{color}
\usepackage{mathrsfs}
\usepackage{amsfonts}
\usepackage{amssymb,amsmath}
\usepackage{CJK}
\usepackage{cite}
\usepackage{cases}
\usepackage{amsthm}

\def\sc{\scriptstyle}
\def\ssc{\scriptscriptstyle}
\def\QED{\hfill$\Box$}
\def\bs{\backslash}
\def\cl{\centerline}
\def\vs{\vspace*}
\def\Z{\mathbb{Z}}
\def\C{{\mathbb{C}}}
\def\Q{{\mathbb{Q}}}
\def\l{\lambda}
\def\WW{{\mathcal W}}
\def\L{{\cal L}}\def\L{{L}}
\def\W{{\cal W}}\def\W{{W}}
\def\ww{{\textbf{\textit{w}}}}

\def\dis{\displaystyle}
\def\D{{\Delta}}

\numberwithin{equation}{section}
\newtheorem{theo}{Theorem}[section]
\newtheorem{defi}[theo]{Definition}
\newtheorem{conv}[theo]{Convention}
\newtheorem{rema}[theo]{Remark}

\newtheorem{lemm}[theo]{Lemma}

\newtheorem{case}{Case}

\allowdisplaybreaks
\begin{document}
\begin{CJK*}{GBK}{song}

\begin{center}
{\large\bf Indecomposable modules of the intermediate series over  $\WW(a,b)$\,$^*$}\footnote {$^*\,$Supported by NSF grant 10825101 of
China\\\indent \ \ \ Corresponding author: Y Xu (xying@mail.ustc.edu.cn) }
\end{center}

\cl{{Yucai Su$^{\,\dag,\,\ddag}$, \  \ \ Ying Xu$^{\,\ddag}$,\ \ \ Xiaoqing Yue$^{\,\dag}$}}

\cl{\small $^{\dag\,}$Department of Mathematics, Tongji
University, Shanghai 200092, China}
\cl{\small
$^{\ddag\,}$Wu Wen-Tsun Key Laboratory of \vs{-2pt}Mathematics}
\cl{\small  University of Science and
Technology of China, Hefei 230026, China} \cl{\small E-mail:
ycsu@tongji.edu.cn,\ xying@mail.ustc.edu.cn, xiaoqingyue@tongji.edu.cn}\vs{5pt}

{\small
\parskip .005 truein
\baselineskip 3pt \lineskip 3pt \noindent{\bf Abstract.} For any complex parameters $a,b$, $\WW(a,b)$  is the Lie algebra with basis $\{\L_i,\W_i\,|\,i\in\Z\}$ and
relations $[\L_{i},\L_j]=(j-i)\L_{i+j}$,
$[\L_i,\W_j]=(a+j+bi)\W_{i+j}$, $[\W_i,\W_j]=0$. In this paper,
indecomposable modules of the intermediate series over  $\WW(a,b)$  are
classified. It is also proved that an irreducible Harish-Chandra $\WW(a,b)$-module
is either a highest/lowest weight module or a uniformly bounded module.
Furthermore, if $a\notin\Q$,
an irreducible weight $\WW(a,b)$-module is simply a $Vir$-module with trivial actions of $W_k$'s.
\vs{5pt}


}
\parskip .001 truein\baselineskip 6pt \lineskip 6pt
\vs{10pt}


%
\section{Introduction}
For any fixed complex numbers $a,b$, there exists a Lie algebra with
basis $\{\L_i,\W_{i}\,|\,i\in\Z\}$ and  Lie brackets
\begin{eqnarray}\label{W-ab}
[\L_{i},\L_j]=(j-i)\L_{i+j},\ \ \ [\L_i,\W_j]=(a+j+bi)\W_{i+j},\ \ \ [\W_i,\W_j]=0.
\end{eqnarray}
This Lie algebra, known as the {\it  algebra $\WW(a,b)$}, is in
fact the semi-direct product $\mathrm{Vir}\ltimes A'_{a,b}$ of $Vir$ by $A'_{a.b}$, where $\mathrm{Vir}={\rm span}\{\L_i\,|\,i\in\Z\}$
is the well-known centerless Virasoro algebra (or
Witt algebra), and $A'_{a,b}$ is a module of the intermediate series defined in \eqref{A-a-b} (which is regarded as an abelian Lie algebra).
Thus, from the definition, one immediately sees that this Lie algebra is closely related to
 the Virasoro algebra and its modules. Due to their extreme importance in mathematics and physics,
representations of the Virasoro algebra (or higher rank Virasoro algebras, e.g., \cite{S1,S2}) have been
 widely studied in the mathematical and physical literatures. For instance,
 a classification of modules of the immediate series over the  Virasoro algebra
 was given in \cite{KS}, unitarizable modules and uniformly bounded modules with composition factors at most two
 were considered respectively in \cite{CP,MP}, and
a classification of Harish-Chandra modules over the Virasoro algebra was presented in \cite{O} (see also, \cite{CP,S0}).

In order to investigate a classification of vertex operator algebras generated by weight $2$ vectors, the $W$-algebra $W(2, 2)$, which is a special case of $\WW(a,b)$ with $a=0,\,b=-1$, was first
introduced and studied in \cite{ZD}. Later on,
a classification of Harish-Chandra modules over $\W(2, 2)$ was considered in \cite{LiuZ}.
Furthermore, the well-known twisted Heisenbeg-Virasoro algebra (without some central elements), whose irreducible Harish-Chandra modules were classified in \cite{LuZ}, is also a special case of $\WW(a,b)$ with $a=b=0$. Thus,
 $\WW(a,b)$ generalizes many meaningful algebras, and it is very natural and desirable to consider representations of
 $\WW(a,b)$.
%
%
%
\begin{defi}\label{defi-1}\rm
A $\WW(a,b)$-module $V$ is calle\vspace*{-3pt}d
\begin{itemize}\parskip-3pt
\item a {\it weight module} if it admits a weight space decomposition
$V=\oplus_{\lambda\in\C} V_{\lambda}$ ($\lambda$ is called a {\it weight} of $V$ in case $V_{\lambda}\neq0$), where \begin{equation}\label{Wei}
V_{\lambda}=\{v_{\lambda}\,|\,L_0 v_{\lambda}=\lambda v_{\lambda}\}\mbox{ \ for \ }\l\in\C;\end{equation}
\item a {\it highest} (resp., {\it lowest} ) {\it weight module} with highest (resp., lowest) weight $\Lambda$ if
there exists $\Lambda\in\C$ such that $V$ is generated by $V_{\Lambda}$, and $V_\lambda=0$ for all $\l\in\C$ with $\l-\Lambda\in\Z_+\bs\{0\}$ (resp., $\Lambda-\l\in\Z_+\bs\{0\}$);
\item
a {\it Harish-Chandra module} if it is a weight module with dim$\,V_{\lambda}<\infty$  for
all $\lambda\in\C$;

\item a {\it uniformly bounded module} {\rm if it is a Harish-Chandra module such that there exists
some $N > 0$ with dim$\,V_{\lambda} \leq N$ for all $\lambda\in\C$;}

\item a {\it module of the intermediate series} {\rm if $V$ is a uniformly bounded module
such that dim$\,V_{\lambda}\leq 1$ for all $\lambda\in\C$}.
\end{itemize}
\end{defi}

The aim of the present paper is to give a classification of (not only irreducible but also) indecomposable modules of the intermediate series over $\WW(a,b)$ and give some description of irreducible weight $\WW(a,b)$-modules.
 Our next goal is to give a classification of all irreducible Harish-Chandra $\WW(a,b)$-modules in some due time. The main techniques used in this paper are developed from that
used in  determining representations of higher rank Virasoro algebras
 \cite{S1,S2}. Similar techniques have been also used in determining representations of
  Block type Lie algebras and  Schr\"{o}dinger-Virasoro algebras \cite{YS,LS}.
We would like to emphasis here that,  as we shall see later on, due to the crucial
fact that the parameter $a$ is not necessarily an integer, modules of the intermediate series over $\WW(a,b)$ may have a rather complicated structure, which is very different from that of the Virasoro algebra, whose indecomposable modules of the intermediate series are
only slightly different than irreducible modules of the intermediate series.

Note that some special cases of $\WW(a,b)$  naturally appear as subalgebras of  many interesting infinite-dimensional
 graded Lie (super)algebras, e.g., the $W$-infinity algebra $\WW_{1+\infty},$ some Block type Lie algebras, $N=2$
 super-Virasoro algebras, Schr\"{o}dinger-Virasoro algebras
\cite{LS,LSZ,S3,YS}, 
etc. Just as   results on  representations of the Virasoro algebra are widely used in classifications of  representations of  Lie (super)algebras which contain the Virasoro algebra as a subalgebra, one can expect that  results on representations of $\WW(a,b)$ may be used in
that of  Lie (super)algebras which contain some $\WW(a,b)$ as a subalgebra
(this is also one of our motivations in presenting the results below).
However, in order to be able to apply our results to  representations of
Lie (super)algebras  which contain $\WW(a,b)$ as a subalgebra, it seems to be necessary to   have a classification of (not only irreducible but also)
indecomposable $\WW(a,b)$-modules of the intermediate series. This is why we consider indecomposable rather than irreducible modules here.

Our first main result in this paper is to give the following classification of
indecomposable $\WW(a,b)$-modules of the intermediate series. Here and below, we always assume $b\ne0$ since $\WW(a,0)\cong\WW(a,1)$ (cf.~\eqref{Imospp}) if $a\notin\Z$, and in case $a\in\Z$, $\WW(a,0)\cong\WW(0,0)$ is simply the
twisted Heisenbeg-Virasoro algebra, whose indecomposable modules of the intermediate series were considered in
\cite{LuZ}.

\begin{theo}\label{theo-2-1}Let $V$ be an indecomposable $\WW(a,b)$-module of the intermediate series. Then we have one of the following:
\begin{itemize}\parskip-3pt
\item[\rm(1)]$V$\,is\,a\,$Vir$-module\,of the\,intermediate series\,$($cf.~\eqref{Vir---}--\eqref{B-a}$)$\,with trivial actions\,of $W_k$'s.
  \item[\rm(2)] $a\notin\Z,\,b\neq0,1$, and $V$ is a sub-quotient \vs{-5pt}of
\begin{equation}\label{Th-eq1}A(\l,\mu),\,\ B(\l,\mu),\,\ A_1(\gamma),\,\ A_2(\gamma),\,\ A_3(\gamma),\,\ B_1(\gamma),\,\ B_2(\gamma)\,{\it or}\ B_3(\gamma)\vs{-5pt},\,\
\phantom{\overline{A}(\l,\mu),}\end{equation} defined in \eqref{W-A-a-b}--\eqref{W-B-r-a-b} for some $\l,\mu\in\C,\,\gamma\in\C\cup\{\infty\}$ $($cf.~Convention $\ref{con1})$.
 \item[\rm(3)] $a\notin\Z,\,b=1$, and  $V$ is a sub-quotient of a module in \eqref{Th-eq1} or a sub-quotient \vs{-5pt}of
\begin{equation}\label{Th-eq1-tilde}
\widetilde{A}(\l,\mu),\,\ \widetilde{B}(\l,\mu),\,\ \widetilde{A}_1(\gamma),\,\ \widetilde{A}_2(\gamma),\,\ \widetilde{A}_3(\gamma),\,\ \widetilde{B}_1(\gamma),\,\ \widetilde{B}_2(\gamma),\,\ \widetilde{B}_3(\gamma)\,{\it or}\ \overline{A}(\l,\mu)\vs{-5pt},\end{equation} defined in \eqref{t-A-lambda-mu}--\eqref{t-B-2-gamm} and \eqref{W-qA-a-b1} for some $\l,\mu\in\C,\,\gamma\in\C\cup\{\infty\}$. Furthermore, $\overline{A}(\l,\mu)$ can occur only when $a\in\Q$.

\item[\rm(4)] $a\in\Z,\,b=1$, and  $V$ is a quotient of $\,\overline{A}(\l,\,\mu,\,c)$ for some $\l,\mu,c\in\C$, which has basis $\{v_m\,|\,m\in\Z\}$ with actions defined \vs{-5pt}by
\begin{eqnarray}\label{ASSSSS}\!\!\!\!\!\!\!\!\!\!\!\!\!\!\!\!\!\!\!\!
&\overline{A}(\l,\,\mu,\,c):&\ \ L_kv_m=(\l+m+\mu k)v_{k+m},\ \ W_kv_m=\delta_{k+a,\,0}c{\sc\,} v_{m+k+a}.
\end{eqnarray}
\end{itemize}
\end{theo}

The second main result is the following description of irreducible weight $\WW(a,b)$-modules.
\begin{theo}\label{theo-irr}\begin{itemize}\parskip-3pt
   \item [\rm(1)]Any irreducible Harish-Chandra $\WW(a,b)$-module
is either a uniformly bounded module or a highest/lowest weight module.
\item[\rm(2)]Let $V$ be an irreducible $\WW(a,b)$-module of the intermediate series. Then we have one of the following:\\ {\rm(i)}
$V$ is a $Vir$-module with trivial actions of $W_k$'s;\\
{\rm(ii)} $a\!\in\!\Q\!\setminus\!\Z,\,b\!=\!1$, and $V$ is a quotient of $\,\overline{A}(\l,\mu)$
defined in \eqref{W-qA-a-b1}  for some $\l,\mu\!\in\!\C$; \\
{\rm(iii)}\,$a\!\in\!\Z$,\,and $V\,$is\,a\,quotient\,of $\overline{A}(\l,{\ssc\!}\mu,{\ssc\!}c)$\,defined\,in\,\eqref{ASSSSS}\,for some $\l,\mu,c\!\in\!\C$\,with $c\!\ne\!0$.

\item[\rm(3)] If $a\notin\Q$, then
an irreducible weight $\WW(a,b)$-module $($not necessarily a Harish-Chandra module$)$ is simply a $Vir$-module with trivial action of $W_k$'s.
\end{itemize}
\end{theo}

The paper is arranged as follows. After presenting some notations, definitions and  preliminary results in Section 2, we first list all possible maximal
indecomposable $\WW(a,b)$-modules of the intermediate series with condition $a\notin\Q$ in Section 3.
Then we give a proof of Theorem \ref{theo-2-1} in Sections 4 and 5 for the case $a\notin\Q$, and in Section 6 for the cases $a\in\Q\bs\Z$ and $a\in\Z$ respectively.
Finally, we give a proof of Theorem \ref{theo-irr} in Section 7.

Throughout the paper, we denote by $\Z,\,\Z_+,\,\Q$ and $\C$ the sets of integers,
nonnegative integers, rational numbers and complex numbers respectively.
%
\section{Preliminaries}From \eqref{W-ab}, one immediately sees that \begin{equation}\label{Vir---}
Vir={\rm span}\{\L_i\,|\,i\in\Z\},\end{equation}
is the well-known centerless
{\it Virasoro algebra}. An indecomposable $Vir$-module of the intermediate series \cite{KS} must be one of $A'_{\l,\mu},
A'(\gamma),B'(\gamma),\, \l, \mu\in\C,\,\gamma\in\C\cup\{\infty\}$ (we add the prime to the notations in order to avoid the confusion with
$\WW(a,b)$-modules to be introduced later),
%
%
%
%
or one of their quotient submodules, where
$A'_{\l,\mu}, A'(\gamma), B'(\gamma)$ all have a basis $\{v_n\,|\,n\in\Z\}$ such
that

\begin{eqnarray}
\label{A-a-b}A'_{\l,\mu}:&\!\!\!\!\!\!&L_iv_j=(\l+j+\mu i)v_{i+j},\\
\label{A-a}A'(\gamma):&\!\!\!\!\!\!&L_iv_0=i(i+\gamma)v_i,\ \ \ \ \ \ L_iv_j=(i+j)v_{i+j}\mbox{ for }j\ne0,\\
\label{B-a}B'(\gamma):&\!\!\!\!\!\!&L_iv_{-i}=-i(i+\gamma)v_0,\ \ L_iv_j=jv_{i+j}\mbox{ for }j\ne-i.
\end{eqnarray}
Here and below, for convenience, we use the following:
\begin{conv}\label{con1}\rm
If $\gamma=\infty$, we always regard $i+\gamma$ as $1$.
Thus in fact $A'(\infty)=A'_{0,1}$ and $B'(\infty)=A'_{0,0}$.
\end{conv}

Denote by
$A''_{0,0}$
%
%
%
%
the unique nontrivial irreducible quotient of $A'_{0,0}$, and $T$  the $1$-dimensional trivial module.

Now let $V$ be an indecomposable $\WW(a,b)$-module of the intermediate series, namely, the weight space $V_\l$ defined by \eqref{Wei} satisfies
${\rm dim}\,V_\l\le1$ for all $\l\in\C$.
Since $\{\W_i\mid i\in\Z\}$ spans an Abelian ideal of $\WW(a,b)$, if  $\W_i V=0$ for all $i\in\Z,$ then $V$ is simply a $Vir$-module, whose structure is well-known. So, we shall always suppose
\begin{equation}\label{weight-1}
\W_iV\ne0\mbox{ \ for some \ }i\in\Z.
\end{equation}
We denote $P(V)=\{\mu\in\C\,|\,V_\mu\ne0\}$, called the {\it set of weights of $V$}.
Fix a weight $\l\in \C$, so that $V_{\l}\ne0$. Using $[\L_0,\L_i]=i\L_i$ and $[\L_0,\W_j]=(a+j)\W_j$, it is easy to obtain that $\L_{j} V_{\l}\in V_{\l+j}$ and $\W_{j} V_{\l}\in V_{\l+a+j}.$
 Since $V$ is indecomposable,  for any two nonzero weight vectors $u,v\in V$ with different weights, there exist weight vectors $v_0:=u,\,v_1,...,v_k:=v$ and basis elements
$x_1,...,x_k\in\{L_i,W_i\,|\,i\in\Z\}$ and nonzero numbers $a_1,...,a_k\in\C$, such that either $x_iv_{i-1}=a_iv_i$ or $a_iv_{i-1}=x_iv_i$, $i=1,...,k$.
This implies
\begin{equation}\label{wei-set}P(V)
\subset\{\l+ja+m\,|\,j,m\in\Z\}.
\end{equation}
In case $a\notin\Q$, we have
\begin{equation}\label{equa-weight}
V=\mbox{\footnotesize$\dis\bigoplus_{-\infty<j<\infty}$} V^j,\
\ \ V^j=\mbox{\footnotesize$\dis\bigoplus_{m\in\Z}$}V_m^j,\ \ \ V^j_m=\{v\in V\,|\,L_0v=(\l+aj+m)v\},\end{equation}
and ${\rm dim\,}V^j_m\le1$.

Let $N_1$ (resp., $N_2$) be the smallest (resp., largest)
integer such that $V^{N_1+1}, V^{N_2-1}\ne0$. Then $N_1< N_2-1$, and $V=\oplus_{N_1<j<N_2} V^j$, such that
each $V^j$ is a nonzero module of the intermediate series over $Vir$ and $W_i V^j\subset V^{j+1}$ for $N_1< j<N_2$.
Note that $N_1$ (resp., $N_2$) can be $-\infty$ (resp., $\infty$).
In case $N_1<-\infty$, we can  suppose $N_1=-1$ if necessary. Define the {\it length} $\ell(V)$ of $V$ to be $N_2-N_1+1$.

In case $a\in{\Q\setminus 0}$, by shifting index of $W_m$ if necessary,
we can suppose $a=\frac{q}{p}$ with $1\le q<p$ and $p,q$ are coprime.
We can also suppose $V=\oplus_{j=0}^N V^j$ with $N<p$ (i.e., $N_1=-1,N_2=N+1\le p$).
We also denote $V^j=0$ if $j<0$ or $j>N$.

\def\M{{\textit{\textbf{V}}}}\def\mm{{\textit{\textbf{v}}}}Let $V$ be a $\WW(a,b)$-module of the intermediate series with decomposition
\eqref{equa-weight}. Take the subspace $\M$ of $V^*$  with
decomposition $\M=\oplus_{-\infty<j<\infty}\M^j$ such that $\M^j=\oplus_{m\in\Z}\M_m^j$ and $\M_m^j=(V_m^j)^*$ (where ``$\,^*\,$'' stands for the ``dual space''). Then $\M$ is a $\WW(a,b)$-module of the intermediate series by defining for $x\in\WW(a,b),\,\mm\in\M$,
\begin{equation}\label{dual-mod}
x \mm(v)=-\mm(xv),\ \forall\,v\in V.
\end{equation}
We simply call $\M$ the {\it dual $\WW(a,b)$-module} of $V$.

In the following sections, we shall determine indecomposable $\WW(a,b)$-modules according to the cases
$a\notin\Q,$ $a\in\Q\backslash\Z$ and $a\in\Z.$
%
\section{Modules of the intermediate series for  case $a\notin\Q$}
In this section, we consider the case $a\notin\Q$. In this case, we can assume $\ell(V)\ge2$, otherwise $V$ is simply
a $Vir$-module. We can also suppose $b\ne0$ because there exist an algebra isomorphism $\eta$: $\WW(a,0)\cong\WW(a,1)$  defined by
\begin{equation}\label{Imospp}
\eta(L_k^{(0)})=L_k^{(1)},\ \ \eta(W_k^{(0)})=(a+k)W_k^{(1)},\ \forall\,k\in\Z,\end{equation}
where the basis elements of $\WW(a,i)$ are denoted by $L_k^{(i)},W_k^{(i)},\,i=0,1$.

We first list all possible maximal indecomposable $\WW(a,b)$-modules of the intermediate series
which we obtained using complex computations (an indecomposable module is {\it maximal} if it cannot be strictly embedded in another indecomposable module).
For any $\l,\mu\in\C$, it is easy to check that there exist indecomposable $\WW(a,b)$ modules (which will be proved to be maximal indecomposable) $A(\l,\mu)$, $B(\l,\mu)$
of the intermediate series with bases \linebreak $\{v_m^j\,|\,j,m\in\Z\}$, $\{v_m^j\,|\,m\in\Z,j=0,1\}$ respectively and actions:
\begin{eqnarray}
\label{W-A-a-b}\!\!\!\!\!\!\!\!\!\!\!\!\!\!A(\l,\mu)&\!\!\!\!\!\!\!\!&={\rm span}\{v_m^j\,|\,j,m\in\Z\}: \nonumber\\&\!\!\!\!\!\!\!\!&
L_kv^j_m=(\l+aj+m+(\mu+jb)k)v^j_{k+m},\ \ W_kv^j_m=v^{j+1}_{k+m}, \ \forall\,j,k,m\in\Z;\\[4pt]
\!\!\!\!\!\!\!\!\!\!\!\!\!\!B(\l,\mu)&\!\!\!\!\!\!\!\!&={\rm span}\{v_m^j\,|\,m\in\Z,\,j=0,1\}: \nonumber\\&\!\!\!\!\!\!\!\!&
L_k v_m^0=(\l+m+\mu k)v_{k+m}^0,W_k v_m^0\!=\!(b(\l+m)-\mu(k+a))v_{m+k}^1,\nonumber\\
\label{W-B-a-b}\!\!\!\!\!&\!\!\!\!\!\!\!\!& L_k v_m^1=(\l+a+m+(\mu+b+1)k)v_{m+k}^1, \ W_k v_m^1=0,\ \forall\,k,m\in\Z.
\end{eqnarray}
Obviously, the dual module of $A(\l,\mu)$ is $A(-\l,1-\mu)$ and the dual module of $B(\l,\mu)$ is $B(-\l-a,-\mu-b)$, by taking the dual basis of $\{v^j_m\,|\,j,m\in\Z\}$ to be $\{(-1)^j w^{-j}_{-m}\,|\,j,m\in\Z\}$, i.e., $w_n^i(v_m^j)=(-1)^j\delta_{i,-j}\delta_{n,-m}$. Thus, no new modules can be produced by taking duals of $A(\l,\mu)$ and $B(\l,\mu)$.

For any $\gamma\in\C\cup\{\infty\}$, we have $\WW(a,b)$-modules $A_1(\gamma)$, $A_2(\gamma)$, $A_3(\gamma)$
with basis \linebreak $\{v_m^j\,|\,j,m\in\Z\}$, $\{v_m^j\,|\,m\in\Z,\,j\geq-1\}$, $\{v_m^j\,|\,m\in\Z,\,j=-1,0\}$ respectively and actions
\begin{eqnarray}\label{W-A-r-a-b-dual}
\!\!\!\!\!\!\!\!A_1(\gamma)&\!\!\!\!\!\!\!\!&={\rm span}\{v_m^j\,|\,j,m\in\Z\}: \nonumber\\&\!\!\!\!\!\!\!\!&\L_kv^0_m=(m+k)v^0_{k+m}\,(m\neq0),\ \ \L_kv^0_{0}=k(k+\gamma)v^0_{k},\nonumber\\
&\!\!\!\!\!\!\!\!&L_kv^j_m=(aj+m+jbk)v^j_{k+m}\,(j\neq0),\nonumber\\
&\!\!\!\!\!\!\!\!&W_kv_m^{-1}=(m+k)v_{m+k}^{0},\ \ W_kv^{0}_m=\delta_{m,0}v^{1}_{k+m},\ \ W_kv^j_m=v^{j+1}_{k+m}\,(j\neq0,-1);\\[4pt]
\!\!\!\!\!\!\!\!\!\!\!\!\!\!A_2(\gamma)&\!\!\!\!\!\!\!\!&={\rm span}\{v_m^j\,|\,j,m\in\Z,\,j\geq-1\}: \nonumber\\&\!\!\!\!\!\!\!\!&L_k v_m^{-1}=(-a\!+\!m\!-\!(1+b)k)v_{m+k}^{-1},\ \ W_k v_m^{-1}=(m+k)(a+(1+b)k+bm)v_{k+m}^0,\nonumber\\
&\!\!\!\!\!\!\!\!&L_k v_m^{0}=(m+k)v_{m+k}^0\,(m\neq0),\ \ L_kv_{0}^0=k(k+\gamma)v_k^0,\ \ W_kv_m^0=\delta_{m,0}v_{m+k}^1,\nonumber\\
&\!\!\!\!\!\!\!\!&L_k v_m^j=(ja+m+jbk)v_{m+k}^j\,(j>0),\ \ W_kv_m^j=v_{m+k}^{j+1}
\,(j\neq-1,0);\label{W-A1-r-a-b-dual}\\[4pt]
\!\!\!\!\!\!\!\!\!\!\!\!\!\!A_3(\gamma)&\!\!\!\!\!\!\!\!&={\rm span}\{v_m^j\,|\,m\in\Z,\,j=0,1\}: \nonumber\\&\!\!\!\!\!\!\!\!&L_k v_m^0=(m+k)v_{m+k}^0\,(m\neq0),\ \ L_kv_{0}^0=k(k+\gamma)v_k^0,\nonumber\\
&\!\!\!\!\!\!\!\!&W_kv_m^0=bv_{m+k}^1\,(m\neq0),\ W_k v_0^0=(b\gamma-a-k)v_k^1,\nonumber\\
&\!\!\!\!\!\!\!\!&\ L_k v_m^1=(a+m+(b+1)k)v_{m+k}^1, \ W_k v_m^1=0.\label{W-B-r-a-b-dual}
\end{eqnarray}
The dual modules of $A_i(\gamma),\,i=1,2,3$ are the following.
\begin{eqnarray}
\!\!\!\!\!\!\!\!\!\!\!\!\!\!B_1(\gamma)&\!\!\!\!\!\!\!\!&={\rm span}\{v_m^j\,|\,j,m\in\Z\}: \nonumber\\&\!\!\!\!\!\!\!\!&
L_k v_m^0=mv_{m+k}^0\,(m\neq -k),\ L_k v_{-k}^0=-k(k+\gamma)v_0^0,\nonumber\\
&\!\!\!\!\!\!\!\!& L_k v_m^j=(ja+m+(jb+1)k)v_{m+k}^j\,(j\neq 0),\nonumber\\
&\!\!\!\!\!\!\!\!&W_k v_m^{-1}=\delta_{k+m,0}v_{k+m}^0,\ \ W_k v_m^0=mv_{k+m}^1,\ \ W_k v_m^{j}=v_{k+m}^{j+1}\,(j\neq 0,-1);\label{W-A-r-a-b}\\[4pt]
\!\!\!\!\!\!\!\!\!\!\!\!\!\!B_2(\gamma)&\!\!\!\!\!\!\!\!&={\rm span}\{v_m^j\,|\,j,m\in\Z,\,j\le1\}: \nonumber\\&\!\!\!\!\!\!\!\!&L_k v_m^0=mv_{m+k}^0\,(m\!\neq\!-k),\ \ L_kv_{-k}^0\!=\!-k(k+\gamma)v_0^0,\nonumber\\
&\!\!\!\!\!\!\!\!& L_k v_m^j=(ja+m+(jb+1)k)v_{m+k}^j\,(j<0),\nonumber\\
&\!\!\!\!\!\!\!\!&L_k v_m^1=(a+m+(b+2)k)v_{m+k}^1,\ \ W_kv_m^{j}=v_{k+m}^{j+1}\,(j<-1),\nonumber\\
&\!\!\!\!\!\!\!\!&W_kv_m^{-1}=\delta_{m+k,0}v_{k+m}^0,\ W_kv_m^0\!=\!m(a+k-bm)v_{m+k}^1, \ W_k v_m^1=0;\label{W-A1-r-a-b}\\[4pt]
\!\!\!\!\!\!\!\!\!\!\!\!\!\!B_3(\gamma)&\!\!\!\!\!\!\!\!&={\rm span}\{v_m^j\,|\,m\in\Z,\,j=-1,0\}: \nonumber\\&\!\!\!\!\!\!\!\!&
 L_k v_m^{0}=mv_{m+k}^0\,(m\neq-k),\ L_kv_{-k}^0=-k(k+\gamma)v_0^0,\ W_kv_m^0=0,\nonumber\\
&\!\!\!\!\!\!\!\!& L_k v_m^{-1}=(-a\!+\!m\!-\!bk)v_{m+k}^{-1}, W_k v_m^{-1}=v_{k+m}^0\,(m\!\neq\!-k),\nonumber\\
&\!\!\!\!\!\!\!\!&W_k v_{-k}^{-1}=(b\gamma-a-k)v_0^0.\label{W-B-r-a-b}
\end{eqnarray}
If $b\ne0,1$, the above are all maximal indecomposable $\WW(a,b)$-modules.
However, when $b=0$ or $1$, there will be more maximal indecomposable modules due to the fact that there exists the algebra  isomorphism \eqref{Imospp}, and thus a $\WW(a,0)$-module $V$ becomes a $\WW(a,1)$-module (denoted by $\tilde V$ when there is no confusion), by the action $xv=\eta^{-1}(x)v$ for all $x\in\WW(a,1)$.

Now suppose $b=1$. We already have maximal indecomposable $\WW(a,1)$-modules $A(\l,\mu)$, $B(\l,\mu)$, $A_i(\gamma)$, $B_i(\gamma)$, $i=1,2,3$. In addition, there will be maximal indecomposable $\WW(a,1)$-modules derived from $\WW(a,0)$-modules by using \eqref{Imospp}.
First we have maximal indecomposable $\WW(a,1)$-modules $\widetilde{A}(\l,\mu)$, $\widetilde{A_1}(\gamma)$, $\widetilde{A}_3(\gamma)$, $\widetilde{B}_1(\gamma)$, $\widetilde{B}_3(\gamma)$ derived from $\WW(a,0)$-modules $A(\l,\mu)$, $A_1(\gamma)$, $A_3(\gamma)$, $B_1(\gamma)$, $B_3(\gamma)$ respectively,
\begin{eqnarray}
\label{t-A-lambda-mu}
\!\!\!\!\!\!\!\!\!\!\!\!\!\!\widetilde{A}(\l,\mu)
&\!\!\!\!\!\!\!\!&={\rm span}\{v_m^j\,|\,j,m\in\Z\}: \nonumber\\
&\!\!\!\!\!\!\!\!&L_kv^j_m=(\l+aj+m+\mu k)v^j_{k+m},\ \ W_kv^j_m=\frac{1}{a+k}v^{j+1}_{k+m};\\[4pt]
\label{t-A-1-gamma}
\!\!\!\!\!\!\!\!\!\!\!\!\!\!\widetilde{A}_1(\gamma)&\!\!\!\!\!\!\!\!&=
{\rm span}\{v_m^j\,|\,j,m\in\Z\}:\nonumber\\
&\!\!\!\!\!\!\!\!&\L_kv^0_m=(m+k)v^0_{k+m}\,(m\neq0),\ \ \L_kv^0_{0}=k(k+\gamma)v^0_{k},\nonumber\\
&\!\!\!\!\!\!\!\!&L_kv^j_m=(aj+m)v^j_{k+m}\,(j\neq0),\ \ W_kv^j_m=\frac{1}{a+k}v^{j+1}_{k+m}\,(j\neq0,-1),\nonumber\\
&\!\!\!\!\!\!\!\!&W_kv_m^{-1}=\frac{m+k}{a+k}v_{m+k}^{0},\ \ W_kv^{0}_m=\frac{\delta_{m,0}}{a+k}v^{1}_{k+m};
\\[4pt]
\label{t-A-3-gamma}
\!\!\!\!\!\!\!\!\!\!\!\!\!\!\widetilde{A}_3(\gamma)&\!\!\!\!\!\!\!\!&=
{\rm span}\{v_m^j\,|\,m\in\Z,j=0,1\}:\nonumber\\
&\!\!\!\!\!\!\!\!&L_k v_m^0=(m+k)v_{m+k}^0\,(m\neq0),\ \ L_kv_{0}^0=k(k+\gamma)v_k^0,\nonumber\\
&\!\!\!\!\!\!\!\!&W_kv_m^0=0\,(m\neq0),\ W_k v_0^0=v_k^1,\nonumber\\
&\!\!\!\!\!\!\!\!&\ L_k v_m^1=(a+m+k)v_{m+k}^1, \ W_k v_m^1=0;
\\[4pt]
\label{t-B-l-gamma}
\!\!\!\!\!\!\!\!\!\!\!\!\!\!\widetilde{B}_1(\gamma)&\!\!\!\!\!\!\!\!&=
{\rm span}\{v_m^j\,|\,m,j\in\Z\}:\nonumber\\
&\!\!\!\!\!\!\!\!&L_k v_m^0=mv_{m+k}^0\,(m\neq -k),\ L_{k}v_{-k}^0=-k(k+\gamma)v_0^0,\nonumber\\
&\!\!\!\!\!\!\!\!& L_k v_m^j=(ja+m+k)v_{m+k}^j\,(j\neq0),\ \ W_k v_m^{-1}=\frac{\delta_{m+k,0}}{a+k}v_{m+k}^0,\nonumber\\
&\!\!\!\!\!\!\!\!&W_k v_m^0=\frac{m}{a+k}v_{m+k}^1,\ \ W_kv_m^j=\frac{1}{a+k}v_{m+k}^{j+1}\,(j\neq0,-1);
\\[4pt]
\label{t-B-3-gamm}
\!\!\!\!\!\!\!\!\!\!\!\!\!\!\widetilde{B}_3(\gamma)&\!\!\!\!\!\!\!\!&=
{\rm span}\{v_m^j\,|\,m\in\Z,j=-1,0\}:\nonumber\\
&\!\!\!\!\!\!\!\!&L_k v_m^{0}=mv_{m+k}^0\,(m\neq-k),\ L_kv_{-k}^0=-k(k+\gamma)v_0^0,\ W_kv_m^0=0,\nonumber\\
&\!\!\!\!\!\!\!\!&L_k v_m^{-1}=(-a\!+\!m)v_{m+k}^{-1}, W_k v_m^{-1}=\frac{1}{a+k}v_{k+m}^0\,(m\!\neq\!-k),\nonumber\\
&\!\!\!\!\!\!\!\!&W_k v_{-k}^{-1}=-v_0^0.
\end{eqnarray}
Note that in case $b=0$, the indecomposable $\WW(a,0)$-module $B(\l,\mu)$ of length $2$ is not maximal,
it is contained in a maximal indecomposable module (denoted by, say, $M$), but the maximal indecomposable $\WW(a,1)$-module derived from $M$ by using \eqref{Imospp} is nothing but the $\WW(a,1)$-module $A(\l,\mu)$. Thus we cannot produce any new maximal indecomposable $\WW(a,1)$-module from the $\WW(a,0)$-module $B(\l,\mu)$.
However, we found a maximal indecomposable $\WW(a,0)$-module of length 2, denoted by $\widehat{B}(\l,\mu)$,
\begin{eqnarray}
\!\!\!\!\!\!\!\!\!\!\!\!\!\!\widehat{B}(\l,\mu)&\!\!\!\!\!\!\!\!&={\rm span}\{v_m^j\,|\,m\in\Z,\,j=0,1\}: \nonumber\\
&\!\!\!\!\!\!\!\!&L_k v_m^0=(\l+m)v_{k+m}^0,\ W_k v_m^0=((am-\l k)\mu+a+k)v_{m+k}^1,\nonumber\\
&\!\!\!\!\!\!\!\!& L_k v_m^1=(\l+a+m+k)v_{k+m}^1, \ W_k v_m^1=0.\nonumber
\end{eqnarray}
Furthermore, the indecomposable $\WW(a,0)$-modules $A_2(\gamma)$, $B_2(\gamma)$ are not maximal, they are contained respectively in the maximal indecomposable  $\WW(a,0)$-modules $\widehat{A}_2(\gamma)$, $\widehat{B}_2(\gamma)$,
\begin{eqnarray}
\!\!\!\!\!\!\!\!\!\!\!\!\!\!\widehat{A}_2(\gamma)&\!\!\!\!\!\!\!\!&={\rm span}\{v_m^j\,|\,j,m\in\Z\}: \nonumber\\
&\!\!\!\!\!\!\!\!&\L_kv^0_m=(m+k)v^0_{k+m}\,(m\neq0),\ \ \L_kv^0_{0}=k(k+\gamma)v^0_{k},\ W_k v_m^{0}\!=\!\delta_{m,0}v_{m+k}^1,\nonumber\\
&\!\!\!\!\!\!\!\!&L_k v_m^j=(ja+m)v_{m+k}^j\,(j>0),\ L_k v_m^j=(ja+m+jk)v_{m+k}^j\,(j<0),\nonumber\\
&\!\!\!\!\!\!\!\!&W_k v_m^{-1}\!\!=\!\!(m+k)(a+k)v_{m+k}^0,\ \  W_kv_m^j\!=\!v_{m+k}^{j+1}\,(j>0),\nonumber\\
&\!\!\!\!\!\!\!\!&W_kv_m^j\!=\!(a+k)v_{m+k}^{j+1}\,(j<-1);
\nonumber
\\[4pt]
\!\!\!\!\!\!\!\!\!\!\!\!\!\!\widehat{B}_2(\gamma)&\!\!\!\!\!\!\!\!&={\rm span}\{v_m^j\,|\,j,m\in\Z\}: \nonumber\\
&\!\!\!\!\!\!\!\!&L_k v_m^0=mv_{m+k}^0\,(m\neq -k),\ L_{k}v_{-k}^0=-k(k+\gamma)v_0^0
,\ W_k v_m^0=m(a+k)v_{m+k}^1,\nonumber\\
&\!\!\!\!\!\!\!\!&L_k v_m^j=(ja+m+k)v_{m+k}^j\,(j<0),\ L_k v_m^j=(ja+m+(j+1)k)v_{m+k}^j\,(j>0),\nonumber\\
&\!\!\!\!\!\!\!\!&W_k v_m^{-1}\!=\!\delta_{m+k,0}v_{m+k}^0,\ W_kv_m^j\!=\!v_{m+k}^{j+1}\,(j<-1),
\ W_kv_m^j=(a+k)v_{m+k}^{j+1}\,(j>1).
\nonumber
\end{eqnarray}
Thus, we have $\WW(a,1)$-modules by using \eqref{Imospp}, denoted by $\widetilde{B}(\l,\mu)$, $\widetilde{A}_2(\gamma)$, $\widetilde{B}_2(\gamma)$,
\begin{eqnarray}
\label{t-B-lambda-mu}
\!\!\!\!\!\!\!\!\!\!\!\!\!\!\widetilde{B}(\l,\mu)&\!\!\!\!\!\!\!\!&={\rm span}\{v_m^j\,|\,m\in\Z,\,j=0,1\}: \nonumber\\
&\!\!\!\!\!\!\!\!&L_k v_m^0=(\l+m)v_{k+m}^0,\ W_k v_m^0=\frac{(am-\l k)\mu+a+k}{a+k}v_{m+k}^1,\nonumber\\
&\!\!\!\!\!\!\!\!& L_k v_m^1=(\l+a+m+k)v_{k+m}^1, \ W_k v_m^1=0;
\\[4pt]
\label{t-A-2-gamm}
\!\!\!\!\!\!\!\!\!\!\!\!\!\!\widetilde{A}_2(\gamma)&\!\!\!\!\!\!\!\!&={\rm span}\{v_m^j\,|\,j,m\in\Z\}: \nonumber\\
&\!\!\!\!\!\!\!\!&\L_kv^0_m=(m+k)v^0_{k+m}\,(m\neq0),\ \ \L_kv^0_{0}=k(k+\gamma)v^0_{k},\ W_k v_m^{0}\!=\!\frac{\delta_{m,0}}{a+k}v_{m+k}^1,\nonumber\\
&\!\!\!\!\!\!\!\!&L_k v_m^j=(ja+m)v_{m+k}^j\,(j>0),\ L_k v_m^j=(ja+m+jk)v_{m+k}^j\,(j<0),\nonumber\\
&\!\!\!\!\!\!\!\!&W_k v_m^{-1}\!\!=\!\!(m+k)v_{m+k}^0, W_kv_m^j\!=\!\frac{1}{a+k}v_{m+k}^{j+1}\,(j>0),
W_kv_m^j\!=\!v_{m+k}^{j+1}\,(j<-1);\\[4pt]
\label{t-B-2-gamm}
\!\!\!\!\!\!\!\!\!\!\!\!\!\!\widetilde{B}_2(\gamma)&\!\!\!\!\!\!\!\!&={\rm span}\{v_m^j\,|\,j,m\in\Z\}: \nonumber\\
&\!\!\!\!\!\!\!\!&L_k v_m^0=mv_{m+k}^0\,(m\neq -k),\ L_{k}v_{-k}^0=-k(k+\gamma)v_0^0
,\ W_k v_m^0=mv_{m+k}^1,\nonumber\\
&\!\!\!\!\!\!\!\!&L_k v_m^j=(ja+m+k)v_{m+k}^j\,(j<0),\ L_k v_m^j=(ja+m+(j+1)k)v_{m+k}^j\,(j>0),\nonumber\\
&\!\!\!\!\!\!\!\!&W_k v_m^{-1}\!=\!\frac{\delta_{m+k,0}}{a+k}v_{m+k}^0,\ W_kv_m^j\!=\!\frac{1}{a+k}v_{m+k}^{j+1}\,(j<-1),\ W_kv_m^j=v_{m+k}^{j+1}\,(j>1).
\end{eqnarray}

The following two sections are devoted to the proof of Theorem \ref{theo-2-1} in case $a\notin\Q$.
Note that for any $j\in\Z$ with $N_1<j<N_2$,
 the $Vir$-module $V^j$  can be in general any of $A'_{\l,\mu}$, $A'(\gamma)$, $B'(\gamma)$, $A''_{0,0}$, $T$, $A''_{0,0}\oplus T$. We split the possible $V^j$ into two cases\vs{-4pt}:\begin{itemize}\parskip-3pt
 \item[(1)]each $V^j$ is an irreducible $Vir$-module of type $A'_{\l,\mu}$.
 \item[(2)]At least one $V^j$ is not an irreducible $Vir$-module of type $A'_{\l,\mu}$.\end{itemize}

\section{Proof of Theorem \ref{theo-2-1} in the first case for $a\notin\Q$}\label{subsect1}
In this section, we always suppose that $a\notin\Q$ and each $V^j$ for $N_1<j<N_2$ is an irreducible $Vir$-module of the intermediate series  of type $A'_{\l,\mu}$
(later on, we shall consider all possible deformations).
Thus we can choose a basis $\{v_m^j\,|\,m\in\Z\}$
of $V^j$ such that

\begin{equation}\label{equa-action1}
L_k v_m^j=\ell^j_{k,m}v_{m+k}^j,\ \ \ell^j_{k,m}=\l+aj+m+k\mu_j,\mbox{ \ and \ } \W_k v_m^j=\ww_{k,m}^j v_{m+k}^{j+1},
\end{equation}
for some $\mu_j,\ww_{k,m}^j\in\C$ (cf.~\eqref{equa-weight}).
\begin{rema}\label{rema-1}\rm
\begin{enumerate}\parskip-3pt\item[(1)]
We use the bold symbol $\ww$ to emphasis that $\ww^j_{k,m}$'s are unknown variables to be determined.
\item[(2)]
We have used the convention that if an undefined symbol technically appears in an expression, we always treat it as zero; for instance, $v^j_m=0$ if $j\le N_1$ or $j\ge N_2$.
\item[(3)] By assumption, we have
\begin{equation}\label{equa-assu}
\l+aj\notin\Z\mbox{ \ or \ }\mu_j\notin\{0,1\}\mbox{ \ for \ }N_1<j<N_2.
\end{equation}
\end{enumerate}
\end{rema}
Fix $j\in\Z$ with
 $N_1<j<N_2-1$.
Since $V$ is indecomposable, there exist some $i_j,m_j$ such that
\begin{equation}\label{equa-notzero}
\ww_{i_j,m_j}^j\ne0.
\end{equation}
Applying $(a+k+bi)W_{i+k}=[L_i,W_k]$ to $v^j_m$, we obtain
\begin{equation}\label{equa-w0}
(a+k+bi)\ww^j_{i+k,m}=\ell^{j+1}_{i,k+m}\ww^j_{k,m}-\ww^j_{k,m+i}\ell^j_{i,m}.
\end{equation}
Applying $(a+k+b(i_1+i_2))[L_{i_1},[L_{i_2},W_k]]=(a+i_2+k+bi_1)(a+k+bi_2)[L_{i_1+i_2},W_k]$
to $v^j_m$, we obtain
\begin{eqnarray}\label{equa-w}
&\!\!\!\!\!\!\!\!\!\!\!\!\!\!\!\!\!\!\!\!\!\!\!\!\!\!&
(a+k+b(i_1+i_2))\Big(\ell^{j+1}_{i_1,i_2+k+m}(\ell^{j+1}_{i_2,k+m}\ww^j_{k,m}-\ww^j_{k,i_2+m}\ell^j_{i_2,m})\nonumber\\
&\!\!\!\!\!\!\!\!\!\!\!\!\!\!\!\!\!\!\!\!\!\!\!\!\!\!&
\phantom{(a+k+b(i_1+i_2))\Big(}
-(\ell^{j+1}_{i_2,i_1+k+m}\ww^j_{k,i_1+m}-\ww^j_{k,i_1+i_2+m}\ell^j_{i_2,i_1+m})\ell^j_{i_1,m}\Big)\nonumber\\
&\!\!\!\!\!\!\!\!\!\!\!\!\!\!\!\!\!\!\!\!\!\!\!\!\!\!&
=(a+i_2+k+bi_1)(a+k+bi_2)(\ell^{j+1}_{i_1+i_2,k+m}\ww^j_{k,m}-\ww^j_{k,i_1+i_2+m}\ell^j_{i_1+i_2,m}).
\end{eqnarray}
\begin{lemm}\label{lemm-w1}
Let $j\in\Z$ with $N_1<j<N_2-1$. 
For any $k\in\Z$, there exist infinitely many values of m such that $\ww^j_{k,m}\ne0$.
\end{lemm}
\noindent{\it Proof.~}~Suppose conversely there exist some $k_0,N\in\Z$ such that $\ww^j_{k_0,m}=0$ for all $m$ with $|m|>N$.
Let $n$ be any integer. Take any $i_1,i_2$ and $k=k_0,\,m=n-i_1$ such that \begin{equation}\label{equa-con-i}
|n-i_1|,|n+i_2|,|n-i_1+i_2|>N.\end{equation}
 Then all terms in \eqref{equa-w} vanish except the term containing $\ww^j_{k_0,n}$, and
we obtain
\begin{equation}\label{equa-w-not=}
(a+k_0+b(i_1+i_2))(\l+a(j+1)+k_0+n+i_2\mu_{j+1})(\l+aj+n-i_1+i_1\mu_j)\ww^j_{k_0,n}=0.
\end{equation}
Using condition \eqref{equa-assu} and the fact that $a\notin\Q$, we can always choose $i_1,i_2$ satisfying \eqref{equa-con-i} such that the coefficient of $\ww^j_{k_0,n}$ in \eqref{equa-w-not=} is not zero. Thus
$\ww^j_{k_0,n}=0$. Take $i=i_0$ and $k=k_0$ in \eqref{equa-w0}, we obtain $\ww^j_{i_0+k_0,m}=0$ if $a+k_0+bi_0\ne0$. Assume $a+k_0+bi_0=0$. Then $b\ne1$ since $a\notin\Q$. We choose $i=i_0+j$, $k=k_0-j$ for any $j\notin\{0,-i_0\}$, we again obtain $\ww^j_{i_0+k_0,m}=0$. This proves $\ww^j_{k,m}=0$ for any $k,m$, a contradiction with \eqref{equa-notzero}.\hfill$\Box$
\begin{rema}\label{rema-simpl}\rm
In order to simplify notation in the following discussions, without loss of generality, we can always suppose $j=0$ by shifting indices $j$ if necessary. When we state the results, we shall take the general case into account.
\end{rema}
\begin{lemm}\label{lemm-mu1} We have the following \vspace*{-4pt}possibilities $($cf.~Remark $\ref{Spapa})$:
\begin{itemize}\parskip-3pt
\item[\rm(1)]
$\mu_{j+1}=\mu_j+b+1$, and $\ww^j_{k,m}=b(\l+ja+m)-(a+k)\mu_j$,
\item[\rm(2)]
$b\neq 1$, $\mu_{j+1}=\mu_j+b$, and $\ww^j_{k,m}=1$,
\item[\rm(3)]
$b=1$, $\mu_j\neq 0$, $\mu_{j+1}=\mu_j+1$, and $\ww^j_{k,m}=1$,
\item[\rm(4)]
$b=1$, $\mu_j=0$, $\mu_{j+1}=1$, and $\ww^j_{k,m}=\frac{cam-c(\l+ja)k+a+k}{a+k}$ for some $c\in\C$,
\item[\rm(5)]$b=1$, $\mu_{j+1}=\mu_j\neq0, 1$, and $\ww^j_{k,m}=\frac1{a+k}$,
\item[\rm(6)]$b=1$, $\mu_{j+1}=\mu_j=0$, and $\ww^j_{k,m}=\frac{cam-c(\l+ja) k+a+k}{(a+k)(\l+ja+a+m+k)}$ for some $c\in\C$,
\item[\rm(7)]$b=1$, $\mu_{j+1}=\mu_j=1$, and $\ww^j_{k,m}=\frac{cam-c(\l+ja) k+a+k}{(a+k)(\l+ja+m)}$ for some $c\in\C$,
\item[\rm(8)]$b=1$, $\mu_j=1$, $\mu_{j+1}=0$, and $\ww^j_{k,m}=\frac{cam-c(\l+ja) k+a+k}{(\l+ja+m)(a+k)(\l+ja+a+m+k)}$ for some $c\in\C$.
\end{itemize}
\end{lemm}
\noindent{\it Proof.~}~By Remark \ref{rema-simpl}, we can suppose $j=0$. Taking $k=0$ and
taking the data $(i_1,i_2,m)$ in \eqref{equa-w} to be $(i,-i,n),\,(i,i,n-i),\,(-i,-i,n+i)$ respectively, we obtain a system of three linear equations on $\ww^0_{0,n-i},\,\ww^0_{0,n},\,\ww^0_{0,n+i}$,
{\small\begin{eqnarray}\label{equa-3equa0}
&\!\!\!\!\!\!\!\!\!\!&
f^{s,-1}_i\ww^0_{0,n-i}+f^{s,0}_i\ww^0_{0,n}+f^{s,1}_i\ww^0_{0,n+i}=0,\ \ \ s=1,2,3,\mbox{ \ \ where,}\\[5pt]
&\!\!\!\!\!\!\!\!\!\!&
f^{1,-1}_i=a(i\mu_0-\l-n)(i(\mu_1-1)+a+\l+n),
\nonumber\\
&\!\!\!\!\!\!\!\!\!\!&
f^{1,0}_i\ \,= a((b(b-1)-\mu_0(\mu_0-1)-\mu_1(\mu_1-1))i^2+2(\l+n)(a+\l+n)),\nonumber\\
&\!\!\!\!\!\!\!\!\!\!&
f^{1,1}_i\ \,=f^{1,-1}_{-i},\nonumber\\
&\!\!\!\!\!\!\!\!\!\!&
f^{2,-1}_i=b(1{\sc\!}+{\sc\!}b{\sc\!}-{\sc\!}2\mu_1(2{\sc\!}+{\sc\!}b
{\sc\!}+{\sc\!}\mu_1))i^3{\sc\!}+{\sc\!}(a(1{\sc\!}-{\sc\!}b(b{\sc\!}+{\sc\!}1)
{\sc\!}+{\sc\!}\mu_1(\mu_1{\sc\!}-{\sc\!}3)){\sc\!}
-{\sc\!}b(\l+n)(3{\sc\!}+{\sc\!}b{\sc\!}-{\sc\!}4\mu_1))i^2
\nonumber\\
&\!\!\!\!\!\!\!\!\!\!&\phantom{f^{2,-1}_i=}
-(a^2-2a(\l+n)(b-1+\mu_1)
-2b(\l+n)^2)i+a(\l+n)(a+\l+n),
\nonumber\\
&\!\!\!\!\!\!\!\!\!\!&
f^{2,0}_i\ \,=2(2bi+a)(i(\mu_0-1)-\l-n)(i\mu_1+\l+n+a),
\nonumber\\
&\!\!\!\!\!\!\!\!\!\!&
f^{2,1}_i\ \,=b(-1{\sc\!}-{\sc\!}b{\sc\!}+{\sc\!}2\mu_0(b{\sc\!}+{\sc\!}\mu_0))i^3 {\sc\!}+{\sc\!}(a(-1{\sc\!}+{\sc\!}\mu_0(\mu_0{\sc\!}+{\sc\!}1){\sc\!}+{\sc\!}2b(-1{\sc\!}+{\sc\!}2\mu_0))
{\sc\!}+{\sc\!}b(\l+n)(b{\sc\!}-{\sc\!}1{\sc\!}+{\sc\!}4\mu_0))i^2
\nonumber\\
&\!\!\!\!\!\!\!\!\!\!&\phantom{f^{2,1}_i\ \,=}
-(a^2+2a(\l+n)(1-b-\mu_1)-2b(\l+n)^2)i+a(\l+n)(a+\l+n)
,\nonumber\\
&\!\!\!\!\!\!\!\!\!\!&
f^{3,-1}_i=f^{2,1}_{-i},
\ \ \ \ \ \ \ \ 
f^{3,0}_i=f^{3,0}_{-i},
\ \ \ \ \ \ \ \ 
f^{3,1}_i=f^{2,-1}_{-i}
 .
\end{eqnarray}
}
Denote the determinant of the coefficients by $\D(i)$, which is zero for all $i$ by Lemma \ref{lemm-w1},
{\small\def\sc{}\begin{eqnarray}\label{eqn-D}
&\!\!\!\!\!\!\!\!\!\!&
0=\D(i)=-a i^6(\mu_1-\mu_0-b)(\mu_1-\mu_0-b-1)(\D_2i^2+\D_1(\l+n)+\D_0), \mbox{ \ \ where,}\\
&\!\!\!\!\!\!\!\!\!\!&
\D_2=4 b^2(-1{\sc\!} +{\sc\!}\mu_0{\sc\!} +{\sc\!} \mu_1)
(-b\mu_0 {\sc\!}+{\sc\!} b^2\mu_0 {\sc\!}+{\sc\!}\mu_0^2{\sc\!} -{\sc\!} \mu_0^3
{\sc\!} +{\sc\!} 2\mu_1{\sc\!}-{\sc\!} b\mu_1{\sc\!} -{\sc\!}
    b^2\mu_1 {\sc\!}+{\sc\!} 2 b\mu_0\mu_1 {\sc\!}-{\sc\!} \mu_0^2\mu_1 {\sc\!}-{\sc\!} 3\mu_1^2 {\sc\!}+{\sc\!}
    \mu_0\mu_1^2{\sc\!} +{\sc\!} \mu_1^3),\nonumber\\
&\!\!\!\!\!\!\!\!\!\!&
\D_1=2ab(b-1)(-2 + b + b^2 + 5\mu_0 - 2 b\mu_0 - 3\mu_0^2 + 7\mu_1 +
    2 b\mu_1 - 6\mu_0\mu_1 - 3\mu_1^2),
\nonumber\\&\!\!\!\!\!\!\!\!\!\!&
\D_0=a^2(2b {\sc\!}-{\sc\!} 3b^2 {\sc\!}+{\sc\!} b^4 {\sc\!}+{\sc\!} 2\mu_0{\sc\!}-{\sc\!}10b\mu_0
{\sc\!}+{\sc\!}
        10 b^2\mu_0 {\sc\!}-{\sc\!} 2 b^3\mu_0 {\sc\!}-{\sc\!} 3\mu_0^2 {\sc\!}+{\sc\!} 10 b\mu_0^2
        {\sc\!}-{\sc\!}        6 b^2\mu_0^2 {\sc\!}-{\sc\!} 2 b\mu_0^3 {\sc\!}+{\sc\!} \mu_0^4 {\sc\!}+{\sc\!} 6\mu_1
\nonumber\\&\!\!\!\!\!\!\!\!\!\!&\phantom{D_{0,0}=}
 {\sc\!}-{\sc\!}10 b\mu_1{\sc\!} + {\sc\!}2 b^2\mu_1
 + 2 b^3\mu_1 - 10\mu_0\mu_1 +
        18 b\mu_0\mu_1 - 6 b^2\mu_0\mu_1 + 2\mu_0^2\mu_1 -
        6 b\mu_0^2\mu_1 + 2\mu_0^3\mu_1
\nonumber\\&\!\!\!\!\!\!\!\!\!\!&\phantom{D_{0,0}=}
         - 11\mu_1^2 +
        8 b\mu_1^2
        + 8\mu_0\mu_1^2 - 6 b\mu_0\mu_1^2 +
        6\mu_1^3 - 2 b\mu_1^3 - 2\mu_0\mu_1^3 - \mu_1^4).\nonumber
\end{eqnarray}}%
Note that for any given $n_0$, if we replace the basis element $v^j_m$ by $v'^j_m=v^j_{m+n_0}$, then this amounts to replacing $\l$ by $\l+n_0$. Thus, \eqref{eqn-D} implies
\begin{equation}\label{equa-mu-j}
\mu_1=\mu_0+b,\mbox{ \ or \ }\mu_1=\mu_0+b+1,\mbox{ \ or \ }\D_2=\D_1=\D_0=0.\end{equation}
The possible solutions of \eqref{equa-mu-j} are
\begin{eqnarray}\label{equa-mu-j-s}
\!\!\!\!\!\!\!\!\!\!\!\!&\!\!\!\!\!\!\!\!\!\!&
\{\mu_1=\mu_0+b\},\{\mu_1=\mu_0+b+1\},
\{b=1,\mu_1=\mu_0\},\{b=1,\mu_1=1-\mu_0\},\nonumber\\
\!\!\!\!\!\!\!\!\!\!\!\!&\!\!\!\!\!\!\!\!\!\!&
\{b=1,\mu_0=\frac{9-t^2}{8},\mu_1=\frac{(t+1)(t+3)}{8}\},
\{\mu_0=\frac{(b+1)(b+2)}{2(2b+1)},\mu_1=\frac{b(1-b)}{2(2b+1)}\},\end{eqnarray}
and \begin{eqnarray}\label{equa-mu-j-s1}\!\!\{\mu_0=0,\mu_1=b+2\},\{\mu_0=1,
\mu_1=b\},\{\mu_0=1-b,\mu_1=0\},\{\mu_0=-1-b,\mu_1=1\}.\end{eqnarray}
Note that when $\mu_j=0$ or $1$, we can always change basis elements $v^j_m$'s such that $\mu_j$ becomes
$1$ or $0$. Thus all four cases in \eqref{equa-mu-j-s1} can be regarded as special cases of the first two cases of
\eqref{equa-mu-j-s}.

Now for all possible cases in \eqref{equa-mu-j-s}, we need to process the \vspace*{-4pt}following.\begin{itemize}
\parskip-3pt\item[(i)] Using Lemma \ref{lemm-w1}, by shifting the index $m$ of basis elements $v^0_m$'s if necessary, we can suppose $\ww^0_{0,0}\ne0$. Then
using \eqref{equa-3equa0}, we can solve for $\ww^0_{0,m}$ (in terms of $\ww^0_{0,0}$).
\item[(ii)]
Then we use \eqref{equa-w0} to solve $\ww^0_{k,m}$.
\item[(iii)] Finally we verify that the solution $\ww^0_{k,m}$ satisfies \eqref{equa-w0}.\end{itemize}

Now we consider six cases in \eqref{equa-mu-j-s} case by case.

\begin{case}\label{case-b+b0}
$\mu_1=\mu_0+b$.
\end{case}
Assume $\ww^0_{0,0}=1$ by rescaling basis elements $v^0_m$'s.
First we suppose 
$b\neq0,1$.
By \eqref{equa-3equa0}, we obtain for all $m\in\Z$, $\ww^0_{0,m}=1$. Taking $k=0$ in \eqref{equa-w0}, we have
\begin{eqnarray}\label{i0==111}
(a+bi)\ww^0_{i,m}\!\!\!&=&\!\!\!(\l+m+a+\mu_1i)\ww^0_{0,m}-(\l+m+\mu_0i)\ww^0_{0,m+i}\nonumber\\
\!\!\!&=&\!\!\!\l+m+a+(b+\mu_0)i-(\l+m+\mu_0i)=a+bi.
\end{eqnarray}
For any $i_0\in\Z$, if $a+bi_0\neq 0$, then $\ww^0_{i_0,m}=1$ for all $m\in\Z$ by \eqref{i0==111}. Suppose
$a+bi_0=0$. If $i_0\ne1$, then $\ww^0_{i,m}=1$ for all $i\neq i_0$. Taking $k=1,i=i_0-1$ in \eqref{equa-w0},
\begin{eqnarray*}
(1-b)\ww^0_{i_0,m}\!\!\!&=&\!\!\!(\l\!+\!m\!+\!1\!+\!a\!+\!(i_0\!-\!1)\mu_1)\ww^0_{1,m}\!-\!(\l\!+\!m\!+\!(i_0\!-\!1)\mu_0)\ww^0_{1,m+i_0-1}\\
&=&\!\!\!\l+m+1+a+(i_0-1)(b+\mu_0)-(\l+m+(i_0-1)\mu_0)\\
&=&\!\!\!1+a+(i_0-1)b=1-b.
\end{eqnarray*}
Thus, $\ww^0_{i_0,m}=1$.
If $i_0=1$, by taking $i=2,k=-1$ in \eqref{equa-w0}, we  again obtain $\ww^0_{1,m}=1$.
Thus Lemma \ref{lemm-mu1}(2) holds.

Now we suppose 
$b=1$.
If $\mu_0\neq 0$, by \eqref{equa-3equa0}, we easily get
$\ww^0_{0,0}-\ww^0_{0,i}=0$, i.e., $\ww^0_{0,m}=1$ for all $m\in\Z$. Then by taking $k=0$ in \eqref{equa-w0},
we  get \begin{eqnarray*}
(a+i)\ww^0_{i,m}\!\!\!&=&\!\!\!(\l+m+a+\mu_1i)\ww^0_{0,m}-(\l+m+\mu_0i)\ww^0_{0,m+i}\\
\!\!\!&=&\!\!\!\l+m+a+(1+\mu_0)i-(\l+m+\mu_0i)=a+i,
\end{eqnarray*}
i.e., $\ww^0_{k,m}=1$ for all $k,m\in\Z$. Thus Lemma \ref{lemm-mu1}(3) holds.

Hence we assume $\mu_0=0$. Then
    \eqref{equa-3equa0} gives
$\ww^0_{0,m-i}-2\ww^0_{0,m}+\ww^0_{0,m+i}=0$. In particular,
 $\ww^0_{0,m+1}-\ww^0_{0,m}$ is a constant. Thus,
 $\ww^0_{0,m}=cm+1$ for all $m\in\Z$
and for some $c\in\C$. Taking $k=0$ in \eqref{equa-w0}, we have
\begin{eqnarray*}
(a+i)\ww^0_{i,m}\!\!\!&=&\!\!\!(\l+m+a+\mu_1i)\ww^0_{0,m}-(\l+m+\mu_0i)\ww^0_{0,m+i}\\
\!\!\!&=&\!\!\!(\l+m+a+i)(cm+1)-(\l+m)(cm+ci+1)
\\\!\!\!&=&\!\!\!
cam-c\l i+a+i.
\end{eqnarray*}
So, $\ww^0_{k,m}=\frac{cam-c\l k+a+k}{a+k}$. It is easy to check that the solution satisfies \eqref{equa-w0} and
Lemma \ref{lemm-mu1}(4) holds.

\begin{case}\label{case-b+b0+1}
$\mu_{1}=b+\mu_0+1$.
\end{case}
Taking $n=0$ in \eqref{equa-3equa0}, we have
\begin{equation*}
(b(i+\l)-a\mu_0)\ww^0_{0,0}-(b\l-a\mu_0)\ww^0_{0,i}=0.
\end{equation*}
Since $\ww^0_{0,0}\neq 0$ by the assumption, we must have $b\l-a\mu_0\ne0$.
We can assume $\ww^0_{0,m}=b(m+\l)-a\mu_0$ by rescaling basis elements $v^0_m$'s. Taking $k=0$ in \eqref{equa-w0}, we have
\begin{eqnarray}\label{equa-case2win}
(a+bi)\ww^0_{i,m}\!\!\!&=&\!\!\!(\l+m+a+\mu_1i)\ww^0_{0,m}-(\l+m+\mu_0i)\ww^0_{0,m+i}\nonumber\\
\!\!\!&=&\!\!\!(\l+m+a+(b+1+\mu_0)i)(b(m+\l)-a\mu_0)\nonumber\\
&&\!-(\l+m+\mu_0i)(b(m+i+\l)-a\mu_0)\nonumber\\
\!\!\!&=&\!\!\!(a+bi)(b(m+\l)-(i+a)\mu_0).
\end{eqnarray}
Thus  Lemma \ref{lemm-mu1}(1) holds for all $k\in\Z$ with $a+bk\neq0$.

Assume $k_0\in\Z$ such that $a+bk_0\neq0$. We first assume $k_0\neq1$. Then $b\neq 1$ since $a\notin\Z$.
Using the fact that $a=-bk_0$ and taking $i=k_0-1,k=1$ in \eqref{equa-w0}, we have
\begin{eqnarray*}
(1-b)\ww^0_{k_0,m}\!\!\!&=&\!\!\!(\l\!+\!m\!+\!1\!+\!a\!+\!(k_0\!-\!1)\mu_1)\ww^0_{1,m}\!-\!(\l\!+\!m\!+\!(k_0\!-\!1)\mu_0)\ww^0_{1,m+k_0-1}\\
\!\!\!&=&\!\!\!(\l+m+1+a+(k_0-1)(b+1+\mu_0))(b(m+\l)-(1+a)\mu_0)\\
\!\!\!&&\!\!\!-(\l+m+(k_0-1)\mu_0)(b(m+k_0-1+\l)-(1+a)\mu_0)\\
\!\!\!&=&\!\!\!(a+bk_0-b+1)(b(m+\l)-\mu_0(k_0+a))\\
\!\!\!&=&\!\!\!(1-b)(b(m+\l)-\mu_0(k_0+a)).
\end{eqnarray*}
Thus Lemma \ref{lemm-mu1}(1) also holds for $k_0$.
Now assume  $k_0=1$. Taking $i=k_0+1,k=-1$ in \eqref{equa-w0}, we again conclude that $\ww^0_{k_0,m}=b(m+\l)-\mu_0(k_0+a)$, and so Lemma \ref{lemm-mu1}(1) holds.

\begin{case}\label{case-b1=}
$b=1$, $\mu_1=\mu_0$.
\end{case}
First assume $\mu_0\neq0,1$. By \eqref{equa-3equa0}, we have $\ww^0_{0,m}=\ww^0_{0,0}=\frac{1}{a}$ by rescaling basis elements $v^0_m$'s. Taking $k=0$ in \eqref{equa-w0}, we have
\begin{eqnarray*}
(a+i)\ww^0_{i,m}\!\!\!&=&\!\!\!(\l+m+a+\mu_1i)\ww^0_{0,m}-(\l+m+\mu_0i)\ww^0_{0,m+i}\\
\!\!\!&=&\!\!\!\frac{\l+m+a+\mu_0i-(\l+m+\mu_0i)}{a}=1.
\end{eqnarray*}
Thus Lemma \ref{lemm-mu1}(5) holds.

Now assume $\mu_0=\mu_1=0$. If we replace the basis elements $v_m^1$'s of $V^1$ by $\tilde v^1_m=\frac{1}{\l+a+m}v^1_m$, then $\mu_1$ becomes $1$, which
becomes a special case considered in Case \ref{case-b+b0} with $b=1$. Thus from Lemma \ref{lemm-mu1}(4), we obtain
Lemma \ref{lemm-mu1}(6) by using the fact that $\tilde v^1_m=\frac{1}{\l+a+m}v^1_m$.
Similarly, we have Lemma \ref{lemm-mu1}(7) if $\mu_0=\mu_1=1$.

\begin{case}\label{case-b1-}
$b=1$, $\mu_1=1-\mu_0$.
\end{case}
First assume $\mu_0\neq 0,1$, and suppose $\ww^0_{0,0}=1$. Taking $n=0$ in \eqref{equa-3equa0}, we easily get
$(\mu_0-1)\mu_0(\ww^0_{0,0}-\ww^0_{0,i})=0.
$
Thus $\ww^0_{0,k}=1$ for all $k\in\Z$. Taking $k=0$ in \eqref{equa-w0}, we have
\begin{eqnarray*}
(a+i)\ww^0_{i,m}\!\!\!&=&\!\!\!(\l+m+a+\mu_1i)\ww^0_{0,m}-(\l+m+\mu_0i)\ww^0_{0,m+i}\\
\!\!\!&=&\!\!\!a+i-2\mu_0 i.
\end{eqnarray*}
Thus, $\ww^0_{k,m}=\frac{a+k-2\mu_0 k}{a+k}$. Using this in \eqref{equa-w0}, we obtain
$\mu_0(2\mu_0-1)=0$, i.e., $\mu_0=\frac{1}{2}$, which is a special case in Case \ref{case-b1=}.

 The case $\mu_0=0$ is a special case of  Case \ref{case-b+b0+1}.
It remains to consider the case $\mu_0=1$. In this case, if we replace the basis elements $v_m^0$'s of $V^0$ by $\tilde v^0_m=(\l+m)v_m^0$, then $\mu_0$ becomes to $0$, thus we obtain Lemma \ref{lemm-mu1}(8) by Lemma \ref{lemm-mu1}(6).

\begin{case}\label{case-b1f}
$b=1$, $\mu_0=\frac{9-t^2}{8}$, $\mu_1=\frac{(t+1)(t+3)}{8}$.
\end{case}
By \eqref{equa-3equa0}, we have
{\small\begin{eqnarray}
&\!\!\!\!\!\!\!\!\!\!&g^1\ww^0_{0,n}-g^2\ww^0_{0,n+i}=0,\mbox{ \ \ where,}\label{equa-b1-0n}\\
&\!\!\!\!\!\!\!\!\!\!&g^1=(t-1)(1+t)(3+t)(16a^2(t-3)-i^2(1+t)^2(t(2+t)-11)+4ai(t-1)(t(2+t)-7)\nonumber\\
&&\ \ \ \ \ \ \ \ +(n+\l)(32a(t-1)-16i(t(2+t)-7))-64(n+\l)^2);\nonumber\\
&\!\!\!\!\!\!\!\!\!\!&g^2=(t-1)(1+t)(3+t)(16a^2(t-3)+i^2(t-3)(5+t)(t(2+t)-11)\nonumber\\
&&\ \ \ \ \ \ \ \ +32a(t-1)(n+\l)-64(n+\l)^2).\nonumber
\end{eqnarray}%
}%
Taking the data $(n,i)$ in \eqref{equa-b1-0n} to be $(n+i,-i)$, we obtain another equation involving $\ww^0_{0,n}$ and $\ww^0_{0,n+i}$. By Lemma \ref{lemm-w1}, the determinant of the coefficients is equal to zero, i.e.,
\begin{eqnarray*}
16(a-i)i^2(a+i)(t-3)(t-1)^2(1+t)^2(3+t)^2(5+t)(t^2\!+\!2t\!-\!7)(t^2\!+\!2t\!-\!11)\!=\!0.
\end{eqnarray*}
Thus, $t=-5,-1,-3,1,3,-1\pm2\sqrt{2},-1\pm2\sqrt{3}$. All these are special cases of Cases \ref{case-b+b0}--\ref{case-b1=}.
%
%
%
%
\begin{case}\label{case-bf}
$\mu_0=\frac{(b+1)(b+2)}{2(2b+1)}$, $\mu_1=\frac{b(1-b)}{2(2b+1)}$.
\end{case}
Similar to Case \ref{case-b1f}, we obtain that the determinant of the coefficients is zero, i.e.,
\begin{eqnarray*}
0\!\!\!&=&\!\!\!\frac{3i^6a^2(a-bi)(a-i-bi)(a+bi)(a+i+bi)}{4(1+2b)^6}\ \times\\
\!\!\!&&\!\!\!(b-1)^2b^2(b+1)^2(b+2)(3b+2)(3b^2+2b+1)(3b^2+4b+2).
\end{eqnarray*}
Thus, $b=-2,-1,0,1,-\frac{2}{3}$, or $3b^2+2b+1=0$ or $3b^2+4b+2=0$. All these except the case $b=-\frac{2}{3}$  are special cases of Cases \ref{case-b+b0}--\ref{case-b1=}.
%
%
%
%
Hence, assume $b=-\frac{2}{3}$. By \eqref{equa-3equa0},
\begin{eqnarray*}\
(a\!{\ssc\!}-{\ssc\!}\!i\!{\ssc\!}-{\ssc\!}\!n\!{\ssc\!}
-{\ssc\!}\!\l)(2a\!+\!i\!+\!n+\!\l)(a\!+\!2(i\!+\!n\!+\!\l))\ww^0_{0,n}
{\ssc\!}\!\!-\!(a\!{\ssc\!}-{\ssc\!}\!n\!{\ssc\!}-{\ssc\!}\!\l)(2a\!+\!n\!+\!\l)(a\!+\!2(n\!+\!\l))\ww^0_{0,n+i}\!=\!0.
\end{eqnarray*}
After solving $\ww^0_{0,n}$, and taking $k=0$ in \eqref{equa-w0}, we obtain the solution for $\ww^0_{k,m}$, which contradicts \eqref{equa-w0}.
This  completes the proof of the Lemma \ref{lemm-mu1}.\hfill$\Box$
\begin{rema}\label{Spapa}\rm Under the condition \eqref{equa-assu}, Lemma \ref{lemm-mu1}(6)--(8) can be regarded as Lemma \ref{lemm-mu1}(4) by re-choosing bases of $V^0,V^1$ as we have seen in the above proof. However, we need to state them separately because we shall need to consider all possible deformations (i.e., dropping condition \eqref{equa-assu}) later.
\end{rema}

If $\ell(V)=2$, by Remark \ref{Spapa} and by considering cases in Lemma \ref{lemm-mu1}(1)--(5), we obtain
that $V$ is a sub-quotient of $A(\l,\mu),\,
B(\l,\mu),\,\widetilde{A}(\l,\mu),\,\widetilde{B}(\l,\mu)$.
%
%
%
%
Hence we have Theorem \ref{theo-2-1}.
Thus in the rest of this section, we assume $\ell(V)\ge3$.

\begin{lemm}\label{lemm-mu-j-mu-j+1}
\begin{itemize}\parskip-3pt
  \item[\rm(1)] If $b\neq1$, then $V$ is a sub-quotient of a module of the form $A(\l,\mu)$.
  \item[\rm(2)] If $b=1,$ then $V$ is a sub-quotient of a module of the form $A(\l,\mu)$ or $\widetilde{A}(\l.\mu)$.
\end{itemize}
%
%
%
%
\end{lemm}
\setcounter{case}{0}
\begin{proof}~Applying $[\W_{k_1},\W_{k_2}]=0$ to $v_m^j$ for $N_1<j<N_2-2$, we obtain
\begin{eqnarray}\label{equa-ww}
\ww^{j+1}_{k_1,k_2+m}\ww^j_{k_2,m}=\ww^{j+1}_{k_2,k_1+m}\ww^j_{k_1,m}.
\end{eqnarray}

\begin{case}\label{ww-bneq1}
$b\neq 0,1$.
\end{case}
By Lemma \ref{lemm-mu1}, there are only two possibilities for $\mu_j$ and $\mu_{j+1}$, i.e, $\mu_{j+1}=\mu_j+b+1$ or $\mu_{j+1}=\mu_j+b$. We want to prove $\mu_{j+1}=\mu_j+b$ for all $j$.
If $\mu_{j+1}=\mu_j+b+1$ and $\mu_{j+2}=b+\mu_{j+1}+1$ for some $j$, then \eqref{equa-ww} gives
\begin{eqnarray*}
&&(b(\l+(j+1)a+k_2+m)-(a+k_1)\mu_{j+1})(b(\l+ja+m)-(a+k_2)\mu_j)\\
&\!\!\!=\!\!\!&(b(\l+(j+1)a+k_1+m)-(a+k_2)\mu_{j+1})(b(\l+ja+m)-(a+k_1)\mu_j).
\end{eqnarray*}
From this, we obtain $b=0$, a contradiction with the assumption.
If $\mu_{j+1}=b+\mu_j+1$ and $\mu_{j+2}=b+\mu_{j+1},$ then \eqref{equa-ww} gives
\begin{eqnarray*}
b(\l+ja+m)-(a+k_2)\mu_j=b(\l+ja+m)-(a+k_1)\mu_j,
\end{eqnarray*}
which implies $\mu_j=0.$ By changing the basis elements $v_m^j$'s, we can take $\mu_j$ to be $1$,
thus $\mu_{j+1}=b+\mu_j$.
If  $\mu_{j+1}=b+\mu_{j}$ and $\mu_{j+2}=b+\mu_{j+1}+1,$  then \eqref{equa-ww} gives
\begin{eqnarray*}
b(\l+(j+1)a+k_2+m)-(a+k_1)\mu_{j+1}=b(\l+(j+1)a+k_1+m)-(a+k_2)\mu_{j+1},
\end{eqnarray*}
which gives $\mu_{j+1}=-b$. Thus, $\mu_j=-2b$ and $\mu_{j+2}=1$. By changing the basis elements $v_m^{j+2}$'s, we can take $\mu_{j+2}$ to be 0, thus  $\mu_{j+2}=b+\mu_{j+1}$. Therefore, Lemma \ref{lemm-mu-j-mu-j+1}(1) holds.

\begin{case}\label{ww-bneq1+}
$b=1$.
\end{case}
By Lemma \ref{lemm-mu1}, for each pair $(j,j+1)$, there are $7$ possibilities, thus in principle, there are $49$ possibilities for the pairs $(j,j+1),\,(j+1,j+2)$. However, by taking \eqref{equa-ww} into account, we have only the following $19$ \vs{-5pt}possibilities.
%
%
%
%
%
\begin{enumerate}\parskip-3pt
  \item If $\mu_{j+1}=\mu_{j}+2$, $\mu_{j+2}=1+\mu_{j+1}$, then $\mu_j=0$, $\ww_{k,m}^j=\l+ja+m$, $\ww_{k,m}^{j+1}=1$.
%
%
%
%

  \item If $\mu_{j}=-1$, $\mu_{j+1}=1$, $\mu_{j+2}=1$, then $\ww_{k,m}^j=\l+(j+1)a+m+k$, $\ww_{k,m}^{j+1}=\frac{1}{\l+(j+1)a+m}$.
  \item If $\mu_{j}=-1$, $\mu_{j+1}=1$, $\mu_{j+2}=0$, then $\ww_{k,m}^j=\l+(j+1)a+m+k$, $\ww_{k,m}^{j+1}=\frac{1}{(\l+(j+1)a+m)(\l+ja+2a+m+k)}$.
  \item If $\mu_{j}=-2$, $\mu_{j+1}=-1$, $\mu_{j+2}=1$, then $\ww_{k,m}^j=1$, $\ww_{k,m}^{j+1}=\l+(j+1)a+m+k$.
  \item If $\mu_{j+1}=\mu_{j}+1$, $\mu_{j+2}=1+\mu_{j+1}$, then $\ww_{k,m}^j=\ww_{k,m}^{j+1}=1$.
  \item If $\mu_{j}=-1$, $\mu_{j+1}=0$, $\mu_{j+2}=1$, then $\ww_{k,m}^j=\ww_{k,m}^{j+1}=1$.
  \item If  $\mu_{j}=-1$, $\mu_{j+1}=0$, $\mu_{j+2}=0$, then $\ww_{k,m}^j=1$, $\ww_{k,m}^{j+1}=\frac{1}{\l+(j+2)a+m+k}$.
  \item If $\mu_{j}=0$, $\mu_{j+1}=1$, $\mu_{j+2}=2$, then $\ww_{k,m}^j=\ww_{k,m}^{j+1}=1$.
  \item If $\mu_{j}=0$, $\mu_{j+1}=1$, $\mu_{j+2}=1$, then $\ww_{k,m}^j=\frac{\l+ja+m}{a+k}$, $\ww_{k,m}^{j+1}=\frac{1}{a+k}$.
  \item If $\mu_{j}=0$, $\mu_{j+1}=1$, $\mu_{j+2}=0$, then $\ww_{k,m}^j=\frac{\l+(j+1)a+m+k}{a+k}$, $\ww_{k,m}^{j+1}=\frac{1}{(a+k)(\l+(j+1)a+m)}$.
  \item If $\mu_{j}=\mu_{j+1}=\mu_{j+2}$, then $\ww_{k,m}^j=\ww_{k,m}^{j+1}=\frac{1}{a+k}$.
  \item If $\mu_{j}=0$, $\mu_{j+1}=0$, $\mu_{j+2}=2$, then $\ww_{k,m}^j=\frac{1}{\l+(j+1)a+m+k}$, $\ww_{k,m}^{j+1}=\l+(j+1)a+m$.
  \item If $\mu_{j}=\mu_{j+1}=0$, $\mu_{j+2}=1$, then $\ww_{k,m}^j=\frac{1}{a+k}$, $\ww_{k,m}^{j+1}=\frac{\l+(j+2)a+m+k}{a+k}$.
  \item If $\mu_{j}=\mu_{j+1}=\mu_{j+2}=0$, then $\ww_{k,m}^j=\ww_{k,m}^{j+1}=\frac{1}{a+k}$.
  \item If $\mu_{j}=\mu_{j+1}=1$, $\mu_{j+2}=2$, then $\ww_{k,m}^j=\frac{1}{\l+ja+m}$, $\ww_{k,m}^{j+1}=1$.
  \item If $\mu_{j}=\mu_{j+1}=\mu_{j+2}=1$, then $\ww_{k,m}^j=\ww_{k,m}^{j+1}=\frac{1}{a+k}$.
 \item If $\mu_{j}\!=\!1,\,\mu_{j+1}\!=\!0,\,\mu_{j+2}\!=\!2$, then $\ww_{k,m}^j\!=\!\frac{1}{(\l+ja+m)(\l+ja+a+m+k)},\,\ww_{k,m}^{j+1}\!=\!\l\!+\!(j\!+\!1)a\!+\!m$.
 \item If $\mu_{j}=1$, $\mu_{j+1}=0$, $\mu_{j+2}=1$, then $\ww_{k,m}^j=\frac{1}{(\l+(j+1)a+m+k)(a+k)}$, $\ww_{k,m}^{j+1}=\frac{\l+(j+1)a+m}{a+k}$.
 \item If $\mu_{j}=1$, $\mu_{j+1}=0$, $\mu_{j+2}=0$, then $\ww_{k,m}^j=\frac{1}{(\l+ja+m)(a+k)}$, $\ww_{k,m}^{j+1}=\frac{1}{a+k}$.
\end{enumerate}

In fact, all those possibilities  are equivalent to the following two possibilities by changing the basis elements $v_m^{j+1},v_m^{j+2}$ if \vs{-5pt}necessary.
\begin{itemize}\parskip-3pt
\item[\rm(1)] $\mu_{j+2}=\mu_{j+1}+1=\mu_{j}+2$, and $\ww_{k,m}^j=\ww_{k,m}^{j+1}=1$.
\item[\rm(2)] $b=1$, $\mu_{j+2}=\mu_{j+1}=\mu_{j}$, and $\ww_{k,m}^j=\ww_{k,m}^{j+1}=\frac{1}{a+k}$.
\end{itemize}
%
%
%
%
%
%
%
%
By induction on $j$, we complete the proof.
\end{proof}

Now we have determined all possible structures of $V$ under the assumption in this section.

\section{Proof of the Theorem \ref{theo-2-1} in the second case for $a\notin\Q$}\label{subsect2}\setcounter{case}{0}
Now we consider the case that $a\notin\Q$ and there exists some $j_0$ such that the $Vir$-module $V^{j_0}$ is either reducible or a composition factor of $A'_{0,0}$. Then all weights of $V^{j_0}$ are integers. For any $j\ne j_0$ with $N_1<j<N_2$, since the weights of $V^j$ are in $a(j-j_0)+\Z$, we see that $V^j$ is an irreducible $Vir$-module of type $A'_{\l,\mu}$. Therefore, such $j_0$ is unique.
By shifting indices if necessary, we can suppose $j_0=0$.
Thus for any $j$ with $N_1<j<N_2-1$ and $j,j+1\ne0$, \eqref{equa-action1} holds with $\l=0$.
From this, we obtain:
\begin{rema}\label{rema-re-1}\rm
All results in Section \ref{subsect1} hold for $j$ with $N_1<j<N_2-1$ and $j,j+1\ne0$.
\end{rema}

We need to consider the following five possible \vspace*{-5pt}cases:\begin{itemize}\parskip-3pt
\item[(a)]$V^0=T=\C v^0_0$ is the $1$-dimensional trivial $Vir$-module,
\item[(b)]$V^0=A''_{0,0}$ is the nontrivial composition factor of $Vir$-module $A'_{0,0}$,
\item[(c)]$V^0=A''_{0,0}\oplus T$,
\item[(d)]$V^0= B'(\gamma)$ or $A'_{0,0}$ (note that $A'_{0,0}$ can be regarded as the case $B(\gamma)$ by
interpreting $i+\gamma$ as $1$ when $\gamma=\infty$ in \eqref{A-a}),
\item[(e)]$V^0=A'(\gamma)$ or $A'_{0,1}$.
\end{itemize}
\subsection{The case $V^0=T$}
\begin{lemm}\label{AM-lll}
If $V^1\ne0$  $($i.e., $N_2>1)$, then $V^{-1}=0$ $($i.e., $N_1=-1)$.\end{lemm}
\noindent{\it Proof.~}~Denote $v^1_m=W_mv^0_0$. Then \begin{equation}\label{equa-V0}
\mbox{$L_iv^1_m=[L_i,W_m]v^0_0=(a+m+bi)v^1_{i+m}$.}\end{equation}
Thus $\mu_1=b$.
If $V^{-1}\ne0$, we suppose $W_iv^{-1}_m=\delta_{i,-m}w_iv^0_0$ for some $w_i\in\C$.
Then for $i\ne k$, we have $0=[W_k,W_i]v^{-1}_{-i}=w_iW_kv^0_0=w_iv^1_k$, i.e., $w_i=0$,
a contradiction with the indecomposable condition on  $V$ (cf.~\eqref{equa-notzero}). Thus $V^{-1}=0$.\QED
\begin{lemm}\label{lemm-muj=jb}If $V^1\ne0$, then
$\mu_j=jb$ for any $0<j<N_2$ and $\ww^j_{i,m}=1$ if $j<N_2-1$. Thus $V$ is a sub-quotient module of $A(0,0)$ $($cf.~\eqref{W-A-a-b}$)$.\end{lemm}
\noindent{\it Proof.~}~We already have $\mu_1=b$. So assume $N_2>2.$ From $0=[W_i,W_k]v^0_0=(\ww^1_{i,k}-\ww^1_{k,i})v^2_{i+k}$, we obtain $\ww^1_{i,k}=\ww^1_{k,i}$. By Lemma \ref{lemm-mu1} (cf.~Remark \ref{rema-re-1}), we obtain $\mu_2=\mu_1+b$ and $\ww^1_{i,m}=1$. Now the result can be proved by induction on $j$.
\QED\vskip5pt

Dually, if $V^{-1}\ne0$ (then $V^1=0$), $V$ is a sub-quotient module of $A(0,1)$.

\subsection{The case $V^0=A''_{0,0}$}
First assume $V^1\ne0$. In this case, we can suppose \eqref{equa-action1} also holds for $j=0$ with $\mu_0=0,$ $v_0^0=0$ and $\ww^0_{k,0}=0,\,k\in\Z$. We claim
\begin{equation}\label{ww----}
W_0v^0_m=\ww^0_{0,m}v^1_{m}\ne0\mbox{ \ for some $m$ with $m\ne0$.}\end{equation}
If not, applying $[L_i,W_0]=(a+bi)W_i$ to $v_m^0$, we obtain $(a+bi)\ww_{i,m}^0=0$, i.e., $\ww^0_{i,m}=0$ for $i\ne-\frac{a}{b}$. If $i_0=-\frac{a}{b}\in\Z$, then $b\ne1$ since $a\notin\Z$. Applying $[L_i,W_{i_0-i}]=(a+i_0-i+bi)W_{i_0}$ to $v_m^0$, we have $(1-b)(i_0-i)\ww_{i_0,m}^0=0$ for $i\in\Z$.
Thus $\ww_{i,m}^0=0$ for all $i,m\in\Z$, a contradiction with \eqref{equa-notzero}. Hence, \eqref{ww----} holds.

\begin{lemm}\label{lemm-da-mu1}If $V^1\neq 0$, we have the following \vs{-5pt}possibilities:
\begin{itemize}\parskip-3pt
\item[\rm(1)]
$\mu_1=b=1$, and $\ww_{k,m}^0=\frac{m}{a+k}$,
\item[\rm(2)]
$\mu_1=b+1$, and $\ww^0_{k,m}=m$,
\item[\rm(3)]
$\mu_1=b+2$, and $\ww^0_{k,m}=m(a+k-bm)$.
\end{itemize}
\end{lemm}
\begin{proof} As in the proof of Lemma \ref{lemm-mu1}, we have \eqref{equa-mu-j-s} or \eqref{equa-mu-j-s1} with $\mu_0=0$. All these cases can be regarded as follows,
\begin{eqnarray}\label{equa-da-0-b}
(b,\mu_0,\mu_1)=(b,0,b),\ \ (b,0,b+1),\ \ (b,0,b+2).
\end{eqnarray}
First consider the case $(b,\mu_0,\mu_1)=(b,0,b)$.
Assume $b\neq0,1$.
By \eqref{equa-3equa0}, we have $\ww_{0,m}^0=\ww_{0,m+i}^0$ for $m\neq0,\pm i$. Thus by rescaling basis elements of $V^0$ if necessary, we can assume $\ww_{0,m}^0=1$ for $m\neq0$. Applying $[L_{-i},[L_i,W_0]]=(a+i-bi)(a+bi)W_0$ to $v_{i}^0$, we obtain
\begin{eqnarray}
(2(a+i)+i(\mu_1-b)(1-\mu_1-b))\ww_{0,i}^0=(a+2i-\mu_1i)\ww_{0,2i}^0,\label{equa-da-i}
\end{eqnarray}
which is a contraction. This proves $b=1$.
By \eqref{equa-3equa0} or as in Case \ref{case-b+b0} in the proof of Lemma \ref{lemm-mu1}, we have $\ww_{0,n-i}^0-2\ww_{0,n}^0+\ww_{0,n+i}^0=0$ for $n\neq 0,\pm i$. Thus, $\ww_{0,n}^0=cn+c'$ for $n\neq 0$ and some $c,c'\in\C$. By \eqref{equa-da-i}, $c'=0$. Applying $[L_i,W_0]=(a+i)W_i$ to $v_m^0$, we have $\ww_{k,m}^0=\frac{am}{a+k}$. By rescaling the basis elements $v_m^{0}$'s, we have $\ww_{k,m}^0=\frac{m}{a+k}$, i.e., Lemma \ref{lemm-da-mu1}(1) holds.

Now consider the case
$(b,\mu_0,\mu_1)=(b,0,b+1)$.
By \eqref{equa-3equa0}, we have $(m+i)\ww_{0,m}^0=m\ww_{0,m+i}^0$ for $m\neq0,\pm i$. Thus, we can assume $\ww_{0,m}^0=m$. Similar to the above, we obtain $\ww_{k,m}^0=m$, i.e., Lemma \ref{lemm-da-mu1}(1) holds.

Finally consider the case
$(b,\mu_0,\mu_1)\!=\!(b,0,b\!+\!2)$.
By \eqref{equa-3equa0}, we have $(m\!+\!i)(a\!-\!b(i\!+\!m))\ww_{0,m}^0$ $=m(a-bm)\ww_{0,m+i}^0$ for $m\neq0,\pm i$. Thus, we can assume $\ww_{0,m}^0=m(a-bm)$. Similar to the above, we have $\ww_{k,m}^0=m(a+k-bm)$,  i.e., Lemma \ref{lemm-da-mu1}(3) holds.
\end{proof}

Now suppose $V^{-1}\neq 0$. In this case,
\eqref{equa-action1} also holds for $j=-1$ with $\mu_0=0,$ $v_0^0=0$ and  $\ww_{k,-k}^{-1}=0,\,k\in\Z$.\setcounter{case}{0}
\begin{lemm}\label{lemm-da-mu-1}If $V^{-1}\neq 0$, we have the following \vs{-5pt}possibilities:
\begin{itemize}\parskip-3pt
\item[\rm(1)]
$\mu_{-1}=-b$, and $\ww_{k,m}^{-1}=1-\delta_{m+k,0}$,
\item[\rm(2)]
$\mu_{-1}=-1-b$, and $\ww^{-1}_{k,m}=(1-\delta_{m+k,0})(a+(1+b)k+bm)$,
\item[\rm(3)]
$b=1,\,\mu_{-1}=0$, and $\ww^{-1}_{k,m}=\frac{1-\delta_{m+k,0}}{a+k}$.
\end{itemize}
\end{lemm}
\begin{proof} By \eqref{equa-mu-j-s} or \eqref{equa-mu-j-s1} with $\mu_0=0$, we have
\begin{eqnarray}\label{equa-da-b-0}
(b,\mu_{-1},\mu_0)=(b,-b,0),(b,-1-b,0),(b,1-b,0).
\end{eqnarray}
First assume
$(b,\mu_{-1},\mu_0)=(b,-b,0)$.
By \eqref{equa-3equa0}, we have $\ww_{0,m}^{-1}=\ww_{0,m+i}^{-1}$ for $m\neq0,\pm i$. Thus, we can assume $\ww_{0,m}^{-1}=1$ for $m\neq0$. Applying $[L_i,W_0]=(a+bi)W_i$ to $v_m^{-1}$ gives $\ww_{k,m}^{-1}=1-\delta_{m+k,0}$, i.e., Lemma \ref{lemm-da-mu-1}(1) holds.

Next assume
$(b,\mu_{-1},\mu_0)=(b,-1-b,0)$.
By \eqref{equa-3equa0}, we have $(a+b(i+m))\ww_{0,m}^{-1}=(a+bm)\ww_{0,m+i}^{-1}$ for $m\neq0,\pm i$. Thus we can assume $\ww_{0,m}^{-1}=a+bm$ for $m\neq0$, and so $\ww_{k,m}^{-1}=(1-\delta_{m+k,0})(a+(1+b)k+bm)$, i.e., Lemma \ref{lemm-da-mu-1}(2) holds.

Now assume
$(b,\mu_{-1},\mu_0)\!=\!(b,1\!-\!b,0)$.
If $b\neq0,1$, then by \eqref{equa-3equa0}, we have $(m\!+\!i)\ww_{0,m+i}^{-1}\!=\!m\ww_{0,m}^{-1}$ for $m\neq0,\pm i$, i.e., $\ww_{0,m}^{-1}=\frac{1}{m}$ for $m\neq0$ by rescaling the basis elements. However, by applying $[L_{-i},[L_i,W_0]=(a+bi)(a+i-bi)W_0$ to $v_{i}^{-1}$, we obtain
\begin{eqnarray*}
(2(i-a)+i(b-\mu_{-1})(b+\mu_{-1}-1))\ww_{0,i}^{-1}=2(-a+i+\mu_{-1}i)\ww_{0,2i}^{-1},
\end{eqnarray*}
which is a contradiction. Thus
$b\!=\!1$.
Then  \eqref{equa-3equa0} gives $(m\!-\!i)\ww_{0,m-i}^{-1}\!-\!2m\ww_{0,m}^{-1}\!+\!(m\!+\!i)\ww_{0,m+i}^{-1}=0$ for $m\neq0,\pm i$. Thus $\ww_{0,m}^{-1}=\frac{cm+c'}{m}$ for $m\neq0$ and some $c,c'\in\C$. Then
 \eqref{equa-3equa0} gives $\ww_{k,m}^{-1}=\frac{ac}{a+k}-\frac{c'}{m+k}$ for $m+k\neq 0$. Applying $[L_i,W_{j}]=(a+j+bi)W_{i+j}$ to $v_{-j}^{-1}$, we have $c'=0$. Thus, $\ww_{k,m}^{-1}=\frac{1-\delta_{m+k,0}}{a+k}$ by rescaling the basis elements, i.e., Lemma \ref{lemm-da-mu-1}(3) holds.
\end{proof}

\begin{lemm}\label{lemm-da}
\begin{itemize}\parskip-3pt
  \item[\rm(1)] If $V^1\neq 0$, then $V^{-1}=0$.
  \item[\rm(2)] 
  $V$ is a sub-quotient module of $B(0,0)$, $B(-a,-b-1)$, $\widetilde{B}(0,\mu)$, $A(0,0)$, $\widetilde{A}(0,0)$, $B_i(\gamma)$, $\widetilde{B}_i(\gamma)$.
\end{itemize}
\end{lemm}
\begin{proof}
In order to obtain Lemma \ref{lemm-da}(1), we consider \eqref{equa-ww} with $j=-1$, and use Lemmas \ref{lemm-da-mu1}, \ref{lemm-da-mu-1}. Then the result can be obtained as in the proof of Lemma \ref{lemm-mu-j-mu-j+1}.

Now Considering all the possibilities in Lemma \ref{lemm-mu1} with $j=1$, and using \eqref{equa-ww} with $j=0$, we obtain Lemma \ref{lemm-da}(2) as in the proof of Lemma \ref{lemm-mu-j-mu-j+1}.
\end{proof}

\subsection{The case $V^0=A''_{0,0}\oplus T$}
As in the proofs of Lemmas \ref{AM-lll}, \ref{lemm-da}, we have $V^1=0$ or $V^{-1}=0$.

\begin{lemm}\label{lemm-daT}\begin{itemize}\parskip-3pt\item[\rm(1)]
If $V^{-1}=0$, $V^{1}\neq 0$, then $\mu_1=b$.
\item[\rm(2)]There does not exist an indecomposable $\WW(a,b)$-module $V$ with $V^{-1}=0$, $V^{1}\neq 0$.
\end{itemize}\end{lemm}
\begin{proof}~We claim\begin{equation}\label{Claim111}\ww^0_{m_0,0}\ne0
\mbox{ \ for  some $m_0\in\Z$}.\end{equation} Otherwise, $V$ would be a decomposable $\WW(a,b)$-module. Applying $[L_i,W_0]-(a+bi)W_i=0$, $[L_{-i},W_i]-(a+i-bi)W_0=0$ to $v_0^0$, we obtain
\begin{eqnarray*}
(a+bi)\ww_{i,0}^0-(a+\mu_1i)\ww_{0,0}^0=0,\ \ (a+i-\mu_1i)\ww_{i,0}^0-(a+i-bi)\ww_{0,0}^0=0.
\end{eqnarray*}
The determinant, denoted by $\D'$,  of coefficients of the above linear equations must be zero, i.e, $\Delta'=i^2(\mu_1-b)(1-\mu_1-b)=0$. Thus, $\mu_1=b$ or $\mu_1=1-b$. By \eqref{equa-w0} with $j=m=0$, we obtain $\mu_1=b$ or $b=1,\mu_1=0$. Hence, by changing the basis elements $v_m^1$'s if necessary, we can always suppose $\mu_1=b$. This proves Lemma \ref{lemm-daT}(1).

As in the proof of Lemma \ref{lemm-da-mu1}, we have $b=1$, and $\ww_{k,m}^0=\frac{am}{a+k}$ for $m\neq0$. Then by \eqref{equa-w0}, we would have $\ww_{k,0}^0=0$ for all $k$, a contradiction with \eqref{Claim111}. This proves
Lemma \ref{lemm-daT}(2).\end{proof}

Dually,  there does exist an indecomposable $\WW(a,b)$-module $V$ with $V^{1}=0$, $V^{-1}\neq 0$.

\subsection{The case $V^0=B'(\gamma)$ or $A'_{0,0}$}
 By Convention \ref{con1}, we can assume $V^0=B'(\gamma)$.
 \begin{lemm}\label{lemm-+aa}
 If $V^1\ne0$, then we have the results in Lemma $\ref{lemm-da-mu1}$.\end{lemm}
\noindent{\it Proof.~}~We claim\begin{equation}\label{MSMcl}
\ww_{0,m_0}^{0}\neq0\mbox{ \ for some $0\neq m_0\in\Z$}.\end{equation}
Otherwise, we may assume $\ww_{0,0}^0=1$
by \eqref{equa-notzero}. As in \eqref{equa-da-0-b}, we only need to consider three cases $\mu_1=b,b+1,b+2$. First assume $\mu_1=b\neq1$, applying $[L_k,W_0]=(a+bk)W_k$ to $v_{m}^0$, $v_{0}^0$, $v_{-k}^0$ respectively, we have $\ww_{k,m}^0=0$, $\ww_{k,0}^0=1$, $(a+bk)\ww_{k,-k}^0=-(a-k+bk)$ for $m,k\neq0$ with $m\neq -k$. Using this results and applying $[L_i,W_j]=(a+j+bi)W_{i+j}$ to $v_{-i-j}^0$, we would obtain a contradiction. Similarly, we would obtain a contradiction for other cases. Thus we have \eqref{MSMcl}. Then similar to the proof of Lemma \ref{lemm-da-mu1}, we obtain the lemma.\hfill\QED

\begin{lemm}\label{lemm-da-b} If $V^{-1}\neq 0$, then $\mu_{-1}=1-b$ and $\ww_{k,m}^{-1}=\delta_{k+m,0}$.
\end{lemm}
\begin{proof} As in \eqref{equa-da-b-0}, we need to consider three cases $\mu_{-1}=-b,-1-b,1-b$. First assume $\ww_{0,m_0}^{-1}\neq0$ for some $0\neq m_0\in\Z$. Similar to the proof of Lemma \ref{lemm-da-mu-1}, we can solve $\ww_{0,m}^{-1}$ for $m\neq 0.$ Then applying $[L_k,W_0]=(a+bk)W_k$ to $v_{m}^{-1}$, $v_{0}^{-1}$, $v_{-k}^{-1}$ respectively with $m\neq -k,0$, we can solve $\ww_{k,m}^{-1}$. Then using \eqref{equa-w0} with $m\neq-i-j,-i,-j,0$ and $m=-i-j,-i,-j,0$ respectively, we would obtain a contradiction for all the cases in Lemma \ref{lemm-da-mu-1}. Thus
 $\ww_{0,m}^{-1}=0$ for $m\neq0$, and we can assume $\ww_{0,0}^0=1$
by \eqref{equa-notzero}. Then as the proof above, we have the lemma.
\end{proof}

Now similar to the proofs of Lemmas \ref{AM-lll}, \ref{lemm-da}, we obtain that $V$ is a sub-quotient module of $B(0,0)$, $\widetilde{B}(0,\mu)$, $B_i(\gamma)$, $\widetilde{B}_i(\gamma)$, $i=1,2,3.$

\subsection{The case $V^0=A'(\gamma)$ or $A'_{0,1}$}
Dually to the previous, we obtain that $V$ is a sub-quotient module of $B(0,1)$, $\widetilde{B}(-a,\mu)$, $A_i(\gamma)$, $\widetilde{A}_i(\gamma)$, $i=1,2,3.$

Now we have determined all possible structures of $V$ under the assumption in this section.
%
\section{Modules of the intermediate series for case $a\in\Q$}\setcounter{case}{0}
First we suppose $a\notin\Z$. Then
we can write \begin{equation}\label{p,q}a=\frac{q}{p}\mbox{ \ with }
p,q\in\Z\bs\{0\},\,p\ge2\mbox{ and $p,q$ are coprime},\end{equation}
 and we can assume $b\ne0$ since $\WW(a,0)\cong\WW(a,1)$.
We can suppose \begin{equation}\label{MAMAss}
V=\bigoplus_{j=N_0}^{N_0+N} V^j,\mbox{ \ and $V^j=0$ if $j<N_0$ or $j>N_0+N$},
\end{equation}
for some $N_0,N\in\Z$ with $1\le N\le p-1$.
%
%
%
%
Note that for any fix $N_0,N\in\Z$ with $1\le N\le p-1$, all modules defined in \eqref{W-A-a-b}--\eqref{t-B-2-gamm}
for the case $a\notin\Q$
remain to be $\WW(a,b)$-modules for the case $a\in\Q$ under the additional conditions:
\begin{equation}\label{adddd}N_0\le j\le N_0+N,\mbox{ \ and  \ } W_kv_m^{N_0+N}=0\mbox{ \ for all }k,m\in\Z.\end{equation}
We use the same symbols to denote these modules. In addition, we have another type of modules, denoted by $\overline{A}(\l,\mu)$, with basis $\{v_m^j\,|\,j,m\in\Z,\,0\leq j\leq p-1\}$ and actions:
\begin{eqnarray}\label{W-qA-a-b1}
\!\!\!\!\!\!\!\!\!\!\!\!\!\!\overline{A}(\l,\mu)&\!\!\!\!\!\!\!\!&={\rm span}\{v_m^j\,|\,j,m\in\Z,\,0\leq j\leq p-1\}: \nonumber\\
&\!\!\!\!\!\!\!\!&L_kv^j_m=(\l+aj+m+\mu k)v^j_{k+m}\ \ (0\leq j\leq p-1),\nonumber\\
&\!\!\!\!\!\!\!\!&W_kv^j_m=\frac{1}{a+k}v^{j+1}_{k+m}\ \ (0\leq j< p-1),\ \ \ W_kv^{p-1}_m=\frac{1}{a+k}v^{0}_{k+m+pa}.
\end{eqnarray}

Now we give a proof of Theorem \ref{theo-2-1} for $a\in\Q\bs\Z$. First assume $W_k V^{N_0+N}=0$ for all $k\in\Z$.
Similar to the proof of Theorem \ref{theo-2-1} for the case $a\notin\Q$, we have the result.
Now assume $W_{k_0}V^{N_0+N}\neq0$ for some $k_0\in\Z$ (by shifting the index, we may assume $N_0=0$). This means $N=p-1$ and  $W_{k_0}V^{p-1}\subset V^{0}$. If we re-denote $V^p=V^0$, then we obtain \begin{equation}\label{mu-p=0}
\mu_{p}=\mu_{0}.\end{equation}
Consider all possibilities, we find out that only in the case $\widetilde{A}(\l,\mu)$, \eqref{mu-p=0} can happen, thus we obtain an extra module $\overline{A}(\l,\mu)$ defined in \eqref{W-qA-a-b1}.
This completes the proof of the theorem in this case.
%

Finally we prove Theorem \ref{theo-2-1} for the case $a\in\Z$.
In this case we suppose $a=0$ since $\WW(a,b)\simeq\WW(0,b)$, and
suppose $b\ne0$ as when $b=0$, the algebra $\WW(0,0)$ is simply the
twisted Heisenbeg-Virasoro algebra, whose indecomposable modules of the intermediate series were considered in
\cite{LuZ}.

The special case $a=0,b=-1$ has also been solved in \cite{LiuZ}. Thus, we  assume $b\neq -1$.
First we suppose $V=V^0={\rm span}\{v_m\,|\,m\in\Z\}$ is a $Vir$-module of type $A'_{\l,\mu}$, and so $L_kv_m=(\l+m+\mu k)v_{m+k}$ and $W_kv_m=\ww_{k,m}v_{k+m}$.
Applying $
[\L_k,\W_0]=bk\W_{k},$ $[\W_k,\W_0]=0$
to $v_m$, comparing the coefficients of $v_{m+k}$, we have
\begin{eqnarray}
&&(\l+m+\mu k)\ww_{0,m}-(\l+m+\mu k)\ww_{0,m+k}=bk\ww_{k,m},\label{azlw}\\
&&\ww_{k,m}\ww_{0,m}=\ww_{0,m+k}\ww_{k,m}.\label{azww}
\end{eqnarray}
If $\ww_{k_0,m_0}\neq 0$ for some $k_0,m_0$ with $k_0\ne0$, then  \eqref{azww} gives $\ww_{0,m_0}=\ww_{0,m_0+k_0}$, and \eqref{azlw} gives $\ww_{k_0,m_0}=0$, which is a contradiction. Thus, $\ww_{k,m}=0$ for all $k\ne0$.
If $b\neq 1,$ by applying $[\L_1,\W_{-1}]=(b-1)\W_0$ to $v_m,$  we obtain
$
\ww_{0,m}=(\l+m-1+\mu)\ww_{-1,m}-(\l+m+\mu)\ww_{-1,m+1}=0.
$
Thus, $W_k$'s act trivially on $V$.
Therefore, we suppose $b=1$. Then \eqref{azlw} gives
\begin{eqnarray*}
(\l+m+\mu k)(\ww_{0,m}-\ww_{0,m+k})=0, \ \ \
(\l+m+k-\mu k)(\ww_{0,m}-\ww_{0,m+k})=0,
\end{eqnarray*}
where the second equation follows from the first by replacing $m,$ $k$ by $m+k,$ $-k$ respectively.
Thus
\begin{eqnarray}\label{zf1}
(2\l+2m+k)(\ww_{0,m}-\ww_{0,m+k})=0.
\end{eqnarray}
Letting $m=0$ in \eqref{zf1} gives $(2\l+k)(\ww_{0,0}-\ww_{0,k})=0.$
If $2\l\notin\Z,$ then $\ww_{0,k}=\ww_{0,0}$ for all $k\in\Z$.
If $2\l\in\Z,$ then $\ww_{0,k}=\ww_{0,0}$ for $k\neq -2\l.$ Letting $m=-2\l$ in \eqref{zf1} gives $(k-2\l)(\ww_{0,-2\l}-\ww_{0,k-2\l})=0$ for all $k\in\Z.$ Thus $\ww_{0,-2\l}=\ww_{0,k-2\l}=\ww_{0,0}.$
Hence, we obtain module $\overline{A}(\l,\,\mu,\,c)$.
Now, the proof of Theorem \ref{theo-2-1} is completed.
%
\section{Irreducible modules of the intermediate series}\setcounter{case}{0}
Now we give a proof of Theorem \ref{theo-irr}.
First we prove Theorem \ref{theo-irr}(3). Thus suppose $a\notin\Q$.
Assume that the set $\{W_k\,|\,k\in\Z\}$ acts nontrivially  on $V$, i.e., there exists some nonzero weight vector
$v_\l\in V_\l$ with weight $\l$ such that $W_kv_\l\ne0$ for some $k$.
Note that we have \eqref{wei-set}.
Denote $V'=\sum_{j\ge1,\,m\in\Z}V_{\l+ja+m}$. It is straightforward to verify that $v_\l\notin V',\,0\ne W_kv_\l\in V'$, and $V'$ is a proper $\WW(a,b)$-submodule of $V$, a contradiciton with the irreducibility of $V$. This proves
Theorem \ref{theo-irr}(3).

Theorem \ref{theo-irr}(2) follows immediately from Theorem \ref{theo-2-1}. It remains to prove
Theorem \ref{theo-irr}(1). Thus assume $V$ is an irreducible Harish-Chandra $\WW(a,b)$-module without highest
and lowest weights.
We can assume $a\in\Q$ which is written as in \eqref{p,q}, otherwise the result follows from Theorem \ref{theo-irr}(3) and Mathieu's Theorem (\cite{O}, Theorem 1).
We claim that for any $m\ne-1,0$ and $\l\in\C$, the linear map
\begin{equation}\label{Inje}
\psi_m:=L_m|_{V_{\l}}\oplus L_{m+1}|_{V_{\l}}\oplus W_m|_{V_{\l}}\oplus W_{m+1}|_{V_{\l}}:\
V_{\l}\to V_{\l+m}\oplus V_{\l+m+1}\oplus V_{\l+a+m}\oplus V_{\l+a+m+1}
\end{equation}
is injective.
\def\ell{k}If not, then there exists some $v_0\in V_\l$ such that $L_mv_0=L_{m+1}v_0=W_{m}v_0=W_{m+1}v_0=0$.
Without loss of generality, we suppose $m>0$. Note that when $\ell\gg0$, we can always express $\ell$ as $\ell=x m+y(m+1)$ for some $x,y\in\Z_+\bs\{0\}$, such that $L_\ell,W_\ell$ can be generated by $L_m,L_{m+1},W_m,W_{m+1}$
(note that when $\ell\gg0$, we either have $W_\ell=\frac1{a+(1-b)m+b\ell}[L_{\ell-m},W_m]$ with ${a+(1-b)m+b\ell}\ne0$ or
$W_\ell=\frac1{a+(1-b)(m+1)+b\ell}[L_{\ell-m-1},W_{m+1}]$ with ${a+(1-b)(m+1)+b\ell}\ne0$). Thus there exists some $K>0$ such that $L_\ell v_0=W_{\ell}v_0=0$ for all $\ell>K$.
Then as the proof of \cite[Proposition 2.1]{S3}, we obtain a highest weight, a contraction with the assumption.
This proves the claim.

Now fix a weight $\l_0\in P(V)$. We have $P(V)\subset \{\l_0+\frac{i}{p}+m\,|\,0\le i\le p-1,\,m\in\Z\}$ by \eqref{wei-set}. Denote $N=\sum_{i=0}^{p-1}\sum_{j=-1}^1{\rm dim\,}V_{\l_0+\frac{i}{p}+j}<\infty$. Then for any $\l=\l_0+\frac{i}{p}+m\in P(V)$, we always have ${\rm dim\,}V_\l\le N$ (which is obvious if $m=0,-1$, and which follows from the injectivity of the map $\psi_{-m}$ in \eqref{Inje} if $m\ne0,1$). Thus $V$ is uniformly bounded,
and the proof of Theorem \ref{theo-irr} is completed.


%
%
%
%

\end{CJK*}
\end{document}